\documentclass[journal]{IEEEtran}

\usepackage{lineno,hyperref}
\usepackage{amsmath}
\usepackage{color}
\usepackage{xcolor}
\usepackage{amsfonts}
\ifCLASSINFOpdf
  \usepackage[pdftex]{graphicx}
  \graphicspath{{Figs/}}
  \DeclareGraphicsExtensions{.pdf,.jpeg,.png}
\else
 \fi
 \usepackage{subcaption}
\newcommand{\R}{\mathbb{R}}
\newcommand{\C}{\mathbb{C}}
\DeclareMathOperator{\sign}{sign}
\usepackage{comment}
\newcommand{\off}[1]{}
\usepackage{subcaption}
\usepackage{array}
\usepackage{amssymb}
\usepackage{amsfonts}
\usepackage{amsthm}
\usepackage[export]{adjustbox}

\begin{document}
\newtheorem{definition}{Definition}
\newtheorem{lemma}{Lemma}
\newtheorem{proposition}{Proposition}
\newtheorem{property}{Property}
\newtheorem{corollary}{Corollary}
\newtheorem{algorithm}{Algorithm}

\title{Revealing Stable and Unstable Modes of Generic Denoisers through Nonlinear Eigenvalue Analysis}

\author{Ester Hait-Fraenkel, and Guy Gilboa,~\IEEEmembership{Member,~IEEE,}
\thanks{Authors are with the Department of Electrical Engineering, Technion, Israel Institute of Technology, e-mail: etyhait@campus.technion.ac.il.}
\thanks{Manuscript received April 19, 2005; revised August 26, 2015.}}

\markboth{Journal of \LaTeX\ Class Files,~Vol.~14, No.~8, August~2015}%
{Shell \MakeLowercase{\textit{et al.}}: Bare Demo of IEEEtran.cls for IEEE Journals}

\maketitle

\begin{abstract}
In this paper, we propose to analyze stable and unstable modes of generic image denoisers through nonlinear eigenvalue analysis.
We attempt to find input images for which the output of a black-box denoiser is proportional to the input.
We treat this as a nonlinear eigenvalue problem. 
 This has potentially wide implications, since most image processing algorithms can be viewed as generic nonlinear operators. 
We introduce a generalized nonlinear power-method to solve eigenproblems for such black-box operators.
Using this method we reveal stable modes of nonlinear denoisers.
These modes are optimal inputs for the denoiser, achieving superior PSNR in noise removal. 
Analogously to the linear case (low-pass-filter), such stable modes are
eigenfunctions corresponding to large eigenvalues, characterized by large piece-wise-smooth structures. 
We also provide a method to generate the complementary, most unstable modes, which the denoiser suppresses strongly. These modes are textures with small eigenvalues.
We validate the method using total-variation (TV) and demonstrate it on the EPLL denoiser (Zoran-Weiss). 
Finally, we suggest an encryption-decryption application.
\end{abstract}


\IEEEpeerreviewmaketitle

\section{Introduction}
\IEEEPARstart{L}{inear} eigenvalue problems are a fundamental tool for theoretical analysis and understanding of linear operators, as well as for various engineering and scientific applications. Thus, extensive studies were dedicated to solving the linear eigenvalue problem, $Lu=\lambda u$, where $L$ is a square matrix of size $n \times n$, $u \in \C^{n \times 1}$ is an eigenvector and $\lambda\in \C$ is an eigenvalue.
Well-known methods for solving eigenvalue problems are the linear power iteration or power method~\cite{mises1929praktische} and the related inverse power method~\cite{golub1996matrix}.
 
Recent interest in nonlinear operators and their image processing applications has led to growing research of the generalized \textit{nonlinear} eigenproblem   
\begin{equation}
\label{eq:EV}
    T(u)=\lambda u,
\end{equation}
where $T(u):\mathbb{R}^{n \times 1}\rightarrow \mathbb{R}^{n \times 1}$ is a bounded non-linear operator in finite dimensions (throughout the paper we will restrict ourselves to the real setting, and will also often refer to eigenvectors as eigenfunctions.). So far, methods assumed $T(u)$ is induced by a convex, one-homogeneous functional $J(u)$, such that $T(u)$ is in the subgradient of $J(u)$: $T(u) \in \partial J(u)$.
Hein and Buhler~\cite{hein2010inverse} extended the inverse power method for non-linear homogeneous functionals (ratio of convex functionals), formulating the iterative scheme as an optimization problem. Nossek and Gilboa~\cite{nossek2018flows} suggested an eigenfunction-generating forward flow for operators $p(u)$, where $p(u) \in \partial J(u)$, with $J(u)$ being a convex, one-homogeneous functional. The evolving signal was smoothed in a series of convex minimization steps. Further theoretical analysis and general algorithms are shown in~\cite{feld2019rayleigh},\cite{aujol2018theoretical},\cite{bresson2013adaptive},\cite{bungert2019asymptotic},\cite{cohen2018energy}. Recently such algorithms are successfully used in semi-supervised learning, combining deep-nets and graph-based label extension \cite{aviles2019beyond}.

However, to the best of our knowledge, until now there has been no attempt to solve and analyze eigenproblems for \textbf{generic} operators. Such operators are very common in signal and image processing, since any nonlinear algorithm (e.g. denoising or deblurring), with same-sized input and output, can be seen as a generic nonlinear operator, with a discrete input $u \in \R^n$. With no analytic operator at hand, such operators are very complex to characterize, and can be treated as black-box operators. 

Nevertheless, such analysis is of great interest, as it reveals the stable and unstable modes of an image denoiser. The stable modes are the optimal or most-suitable inputs for the denoiser, achieving superior PSNR in noise removal. These are the eigenfunctions corresponding to large eigenvalues. On the contrary, the unstable modes are the least-suitable inputs, which are strongly suppressed. These are the eigenfunctions corresponding to small eigenvalues.

In this paper, we suggest a generalized method to solve and analyze eigenproblems for generic non-linear denoising operators, in order to reveal their stable and unstable modes. We show similar interpretations of linear concepts in the adapted generic nonlinear setting. For example, large eigenfunctions of denoisers, with large, highly stable structures, are equivalent to low frequency components for linear smoothing operators. We adapt the well known power iteration to generate eigenfunctions for these black-box operators. We establish theoretical requirements for convergence and show steady state properties of our framework, and provide analysis tools for validation. 
We showcase our results using the black-box generic EPLL denoiser~\cite{zoran2011learning}, including robustness properties.
The TV denoising~\cite{rudin1992nonlinear} is used for verification, being a well-studied, theoretically-established, functional-based method.
Finally, we suggest an application of a decryption-encryption scheme.
\subsubsection*{\textbf{Contributions and Novelties}}
\begin{itemize}
    \item \textbf{We formulate the question of optimal inputs for a denoiser as a nonlinear eigenvalue problem for generic, black-box operators.} We handle implicit image-processing algorithms, where no analytic operator is at hand. We do not restrict the discussion to functional-induced operators, such as in~\cite{hein2010inverse},\cite{nossek2018flows}. 
    \item \textbf{We propose an algorithm for finding eigenfunctions for such generic operators}. We generalize known linear concepts into a nonlinear framework. Specifically, we show that under a Lipschitz assumption of the operator, this method converges to an eigenfunction.
\end{itemize}
The rest of the paper is organized as follows. Sec. \ref{Sec::theory} presents the method and a theoretical analysis of its properties and behavior. It also shows validation measures, useful induced operators and a method to generate a series of eigenfunctions. Sec. \ref{Sec::res} shows experimental results for different operators and suggests an application. Sec. \ref{Sec::conc} concludes our work.

\section{Theory and Analysis}\label{Sec::theory}
In this section, we present the theory and analysis of the generalized power iteration algorithm as adapted for \textbf{generic, nonlinear} operators.
We denote the evolving signal as $u^k \in \mathbb{R}^{n \times 1}$, and the generic, nonlinear operator used as $T(u): \: \mathbb{R}^{n \times 1} \to \mathbb{R}^{n \times 1}$. $\|\cdot\|$ denotes the $\ell^2$-norm.

\subsection{The Linear Case}
We shortly display well-known basic concepts of the linear eigenproblem $Lu=\lambda u$, where $\lambda\in \mathbb{R}$ and $u \in \mathbb{R}^{n \times 1}$ are the eigenvalue and eigenfunction, respectively, of the linear operator $L\in \mathbb{R}^{n \times n}$. These concepts will be generalized and extended in our work for generic non-linear problems. 

An elementary, widely used method to solve this problem is the \emph{linear power iteration}~\cite{mises1929praktische}. In this method, an initial guess $f$ is gradually evolved into the eigenfunction corresponding to a large eigenvalue, by iteratively applying the operator and normalizing the result: 
\begin{equation}
\label{eq:LinPM}
  u^{k+1}=\frac{Lu^k}{\|Lu^k\|},  
\end{equation}
initialized with $u^0=f$, $k=1,2,...$. 
For linear operators, its convergence to an eigenfunction can be easily shown (e.g.~\cite{trefethen1997numerical}), with a convergence rate depending on the ratio between the two largest eigenvalues. 

Power iteration extensions allow finding additional eigenfunctions. First, to generate a small eigenvalue, power iterations of \eqref{eq:LinPM} with an alternative operator, $L^\dagger \equiv I-\alpha L$, ~\cite{pohlhausen1921berechnung} yield its large eigenvector, which is a small one for $L$  (also see~\cite{golub1996matrix}).
Second, deflation or projection methods (e.g.~\cite{wilkinson1965algebraic}) generate \textit{more} eigenpairs, relying on previously found eigenvectors of larger eigenvalues. In projection methods, the signal is iteratively projected to the space orthogonal to the known eigenvectors. The Arnoldi iteration~\cite{arnoldi1951principle} uses the Gram-Schmidt process to extract an orthonormal basis, approximating a few large eigenvectors.  

Finally, the well-known Rayleigh quotient~\cite{horn2012matrix} is defined for real symmetric matrices
as:
\begin{equation}
    \label{eq:Lin_R}
    R_{lin}(u):=\frac{u^T L u}{u^T  u}= \frac{\langle u,Lu\rangle}{\|u\|^2}.
\end{equation}
Its Euler-Lagrange equation yields the eigenvalue problem, thus any eigenvector is a critical point of the Rayleigh quotient. Moreover, the Rayleigh quotient can be understood as a generalized or approximated eigenvalue computed for any $u$ (not just eigenvectors). For eigenvectors, the Rayleigh quotient is exactly the corresponding eigenvalue $\lambda$. These methods and measures will be later generalized into the nonlinear setting.

\subsection{Method: Nonlinear Generic Power Iteration} \label{Sec::method}
We adapt the \textit{linear} power iteration to
generate eigenfunctions for \textit{non-linear, generic} operators. An initial signal (in our case, image) $f$ evolves using the power iteration method to generate the process $\{u^k\}$:
\begin{algorithm}{Power iterations with a generic operator $T(\cdot)$.}
\label{AlgPower}
\begin{enumerate}
    \item Initialization: $u^0 \gets f/\|f\|$, $\,\,k \gets 1.$
    \item Repeat until $k=K$ or $\|u^{k+1}-u^k\|< \varepsilon$:\\ 
    $u^{k+1} \gets \sign(\langle u^k,T(u^k)\rangle)\frac{T(u^k)}{\|T(u^k)\|}$, $\,\,\,k \gets k+1.$
\end{enumerate}
\end{algorithm}
For power iterations to be well defined, we assume that $\forall k$, $0 < \|T(u^k)\| < \infty$, $T(u^k) \ne 0$, $\langle u^k,T(u^k)\rangle \ne 0$.
The result depends on several factors: the initial condition $f$, the number of iterations, and the operator $T(u)$ and its different parameters (e.g. patch distribution model, estimated noise or estimated blur kernel). We demonstrate these factors (Sec. \ref{Sec::res}) and even exploit them for our decryption-encryption application. Though our analysis addresses Algorithm \ref{AlgPower}, in practice we use a slightly modified version. First, we prevent evolution to a trivial constant eigenfunction ($\lambda=\lambda_{max}=1$ for denoisers) by removing the mean value, such that the signal is of zero mean but of constant variance, preventing a loss of contrast. Second, to handle operators with a desired value range, we modify the normalization stage.
\begin{algorithm}{Power iterations for non-trivial solutions.}
\label{AlgPower2}
\begin{enumerate}
    \item Initialization: $f^0 \gets f^0-\overline{f^0}$, $\,\,k \gets 1.$
    \item Repeat until $k=K$ or $\|u^{k+1}-u^k\|< \varepsilon$:\\ 
    $u^{k+1} \gets T(u^k)$\\
    $u^{k+1} \gets u^{k+1}-\overline{u^{k+1}}$\\
    $u^{k+1} \gets \frac{u^{k+1}}{\|u^{k+1}\|}\|f^0\|$, $\,\,\,k \gets k+1.$
\end{enumerate}
\end{algorithm}

\subsection{Steady State Properties}
We analyze the steady state behavior of the process $\{u^k\}$: the relation between convergence and reaching an eigenfunction, and eigenvalue characteristics. Note that at this point, we make minimal assumptions regrading the nature of $T(u)$.
Thus, our observations can be applied to any generic, black-box operator. General mild assumptions, valid for any reasonable image-processing algorithm, are the existence and non-triviality of $T(u)$ $\forall u$ (e.g. mapping any function to a constant).
The existence of eigenfunctions is a broad topic and cannot be proved for any general operator. However, as we show here, convergence immediately implies existence.
\begin{lemma}\label{lemma}
$\forall k, \|u^k\|=1$.
\end{lemma}
\begin{proof}
Trivial, as each iteration is normalized by the norm:
\\
$\|u^k\|^2=\langle u^k,u^k \rangle = \langle \sign(\langle u^{k-1},T(u^{k-1})\rangle)\frac{T(u^{k-1})}{\|T(u^{k-1})\|},
\\
\sign(\langle u^{k-1},T(u^{k-1})\rangle)\frac{T(u^{k-1})}{\|T(u^{k-1})\|} \rangle = 1.$
\end{proof}

\begin{proposition}\label{prop1}
Algorithm \ref{AlgPower} converges after a finite number of steps: $\exists N$ s.t. $\forall k>N,$ $u^{k+1}=u^k$, \textbf{if and only if} $u^k$ solves the eigenproblem \eqref{eq:EV}.
\end{proposition}

\begin{proof}
First, if $\exists N$ such that $\forall k>N,$ $u^{k+1}=u^k$, then $u^k=u^{k+1}=\sign(\langle u^k,T(u^k)\rangle)\frac{T(u^k)}{\|T(u^k)\|}$.
Thus, $T(u^k)=\sign(\langle u^k,T(u^k)\rangle)\|T(u^k)\|u^k = \lambda u^k$, where 
$\lambda = \sign(\langle u^k,T(u^k)\rangle)\|T(u^k)\|$.

Second, if some iterate $u^k$ admits \eqref{eq:EV}, then by taking the inner product with $u^k$ on both sides we reach $\sign(\langle u^k,T(u^k)\rangle)=\sign(\lambda)$.
Using the definition of $u^{k+1}$ and Lemma \ref{lemma} we get $u^{k+1}=\sign(\langle u^k,T(u^k)\rangle)\frac{T(u^k)}{\|T(u^k)\|}=$
$\sign(\lambda)\frac{\lambda u^k}{\|\lambda u^k\|}=u^k$. Thus, the process converges. 
\end{proof}

Several reasonable assumptions can be made regarding the nature of $T$:
\begin{enumerate}
    \item If $T$ is a denoising (coarsening) operator, then $\forall k,$ $\|u^k\|\geq\|T(u^k)\|$. \label{assump1}
    \item If $T$ is a deblurring (sharpening) operator, then $\forall k,$ $\|u^k\|\leq\|T(u^k)\|$. \label{assump2}
    \item $\forall k,$ $\langle u^k,T(u^k)\rangle \geq 0$. Positive correlation between input and processed input (typical in denoising). \label{assump3}
    \item $\forall k,$ $\langle u^k,T(u^k)\rangle \leq 0$.  Negative correlation (rare). \label{assump4}
\end{enumerate}
\begin{corollary}{Eigenvalue range.}
\label{cor_lam}
If $u^k$ is a converged solution of Algorithm \ref{AlgPower}, that is, $T(u^k)=\lambda u^k$, then the following holds w.r.t. the eigenvalue $\lambda$:
\begin{enumerate}
\item $\lambda=\sign(\langle u^k,T(u^k)\rangle)\|T(u^k)\|$.
\item If assumptions \ref{assump1}, \ref{assump3} hold, then $0 \leq\lambda\leq 1$. If assumptions \ref{assump1}, \ref{assump4} hold, then $-1 \leq\lambda\leq 0$. That is, for a denoising operator, $|\lambda|\leq 1$.
    \item If assumptions \ref{assump2}, \ref{assump3} hold, then $\lambda\geq 1$. If assumptions \ref{assump2}, \ref{assump4} hold, then $\lambda\leq -1$. That is, for a deblurring operator, $|\lambda|\geq 1$.
\end{enumerate}
\end{corollary}
The above can be easily verified using the proof of Proposition \ref{prop1}, Lemma \ref{lemma} and the respective assumptions.
For example, to show the first assertion of point (2) above: Assumption 1 and Lemma \ref{lemma} yield $1=\|u\| \ge \|T(u)\|=|\lambda|$, whereas from point (1) above, proved in the proposition, and Assumption 3 we have: $\sign(\lambda)=\sign(\langle u^k,T(u^k)\rangle) \ge 0$, thus we conclude
$0 \le \lambda \le 1$.

\begin{figure}
\captionsetup[subfigure]{justification=centering}
\centering
\begin{subfigure}{0.5\textwidth}
    \centering
    \includegraphics[width=0.95\textwidth]{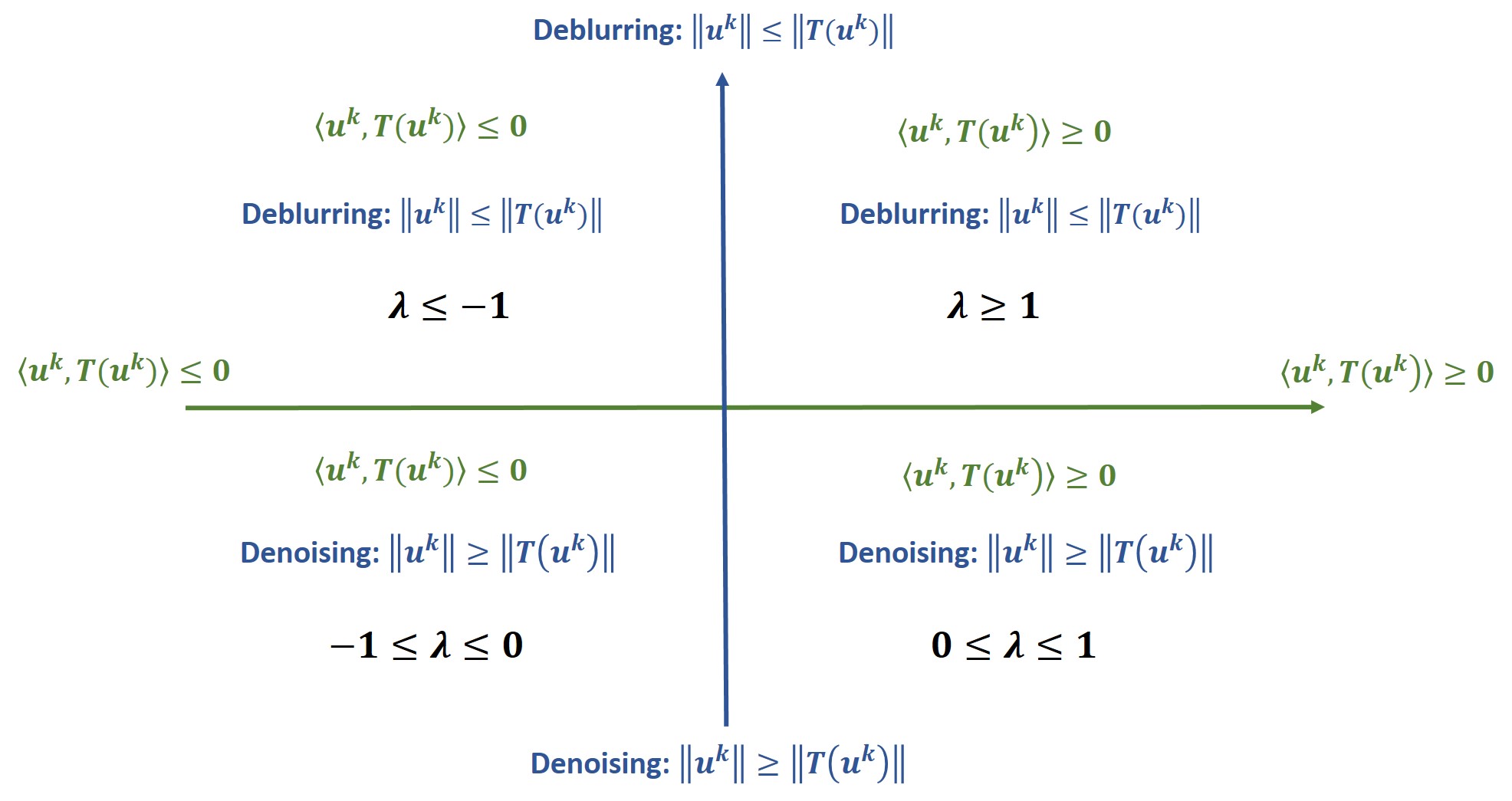}
    \caption{Corollary \ref{cor_lam}: eigenvalue sign and magnitude}
\end{subfigure}
\\
\begin{subfigure}{0.4\textwidth}
    \centering
    \includegraphics[width=0.4\textwidth]{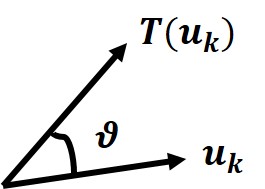}
    \caption{Definition \ref{def2}: angle between $u^k$ and $T(u^k)$}
    \label{Fig::angle}
\end{subfigure}
\caption{Illustrations of Corollary \ref{cor_lam}, Definition \ref{def2}.}
\end{figure}

%
%

\subsection{Eigenfunction Validation Measures}\label{Sec::measures}
Since the suggested method for eigenfunction generation lacks any analytic solutions, we are obliged to verify our results. Thus, we propose several easy to calculate validation measures, in order to verify the convergence of our method to eigenfunctions (up to some minor error). We discuss their behavior throughout the power iteration and in steady state, and show their usefulness for validating that a specific signal is an eigenfunction.

{\bf Point-wise validation and visualization.}
At the risk of stating the obvious, the most fundamental validation measure for an eigenfunction $u$ is that it must admit: $T(u)=\lambda u$, for some constant $\lambda \in \mathbb{R}$. Specifically for images (which are our focus here), each pixel $(i,j)$ must hold $\lambda = T(u)_{ij}/u_{ij}$.
Thus, when visualizing the ratio $T(u)/u$, we should ideally obtain a constant image. This visualization can also point out image regions of less accuracy, e.g. near zero values, where the denominator of the measure is less stable.
It should also hold for a specific image row/ column. This is obviously also a simple way of finding the corresponding eigenvalue.

{\bf Global indicators.} First, we generalize the Rayleigh quotient~\cite{horn2012matrix} for non-linear operators:
\begin{definition}\label{def1}
The Rayleigh quotient of $u^k$ is defined as
\begin{equation}
\label{eq:R}
 R(u^k)=\frac{\langle u^k,T(u^k) \rangle}{\|u^k\|^2}.   
\end{equation}
\end{definition}
The measure was also defined and investigated for nonlinear operators induced by a one-homogeneous functional $J(u)$~\cite{nossek2018flows} as follows: $R(u)=\frac{J(u)}{\|u\|^2}$.
The definition of \eqref{eq:R} generalizes \cite{nossek2018flows}, since for one-homogeneous functionals $J(u)=\langle u, p(u) \rangle$, where $p(u)$ is a subgradient element of $J(u)$, and $T(u)=p(u)$. It also naturally generalizes the linear case \eqref{eq:Lin_R}, by setting $T(u)=L(u)$.
\begin{proposition} \label{prop3}
$\forall k,$ $|R(u^k)| \leq \|T(u^k)\|$. Equality holds \textbf{if and only if} $u^k$ is an eigenfunction, admitting \eqref{eq:EV}.
\end{proposition}

\begin{proof}
Using Lemma \ref{lemma} and the Cauchy-Schwarz inequality:

$|R(u^k)|=|\frac{\langle u^k,T(u^k)\rangle)}{\|u^k\|^2}|=|\langle u^k,T(u^k)\rangle)| \leq \|u^k\|\|T(u^k)\|=\|T(u^k)\|$.\\
The Cauchy-Schwarz then holds in equality if and only if $u^k$ and $T(u^k)$ are linearly dependent, that is, $u^k$ is a solution of the eigenproblem \eqref{eq:EV}.
\end{proof}
\begin{proposition}\label{prop5}
Suppose that exactly at iteration $k=N$ Algorithm \ref{AlgPower} converged.  
Then $\forall k<N,$ $|R(u^k)|<\|T(u^k)\|,$ $\forall k\geq N,$ $|R(u^k)|=\|T(u^N)\|$.
\end{proposition}

\begin{proof}
Follows directly from Propositions \ref{prop1} and \ref{prop3}.
%
\end{proof}
Since the algorithm converges when the Rayleigh quotient stabilizes, its stabilization is a good indication for convergence. Then the value of the Rayleigh quotient is also the eigenvalue. In practice, for the operators tested, the Rayleigh quotient monotonically increases to the eigenvalue (see Sec. \ref{Sec::res}).

Another global measure is the angle between $u^k$ and $T(u^k)$ (Fig. \ref{Fig::angle}), inspired by a similar definition in~\cite{nossek2018flows} for their subgradient-induced operator.  

\begin{definition}\label{def2}
The cosine of the angle between $u^k$ and $T(u^k)$ is defined as
\begin{equation}
\label{eq:cos_theta}
 cos\theta=\frac{\langle u^k,T(u^k)\rangle)}{\|u^k\|\|T(u^k)\|}.
\end{equation}
\end{definition}

\begin{proposition}\label{prop4}
The angle between $u^k$ and $T(u^k)$ is $\pi n, n \in Z$ \textbf{if and only if} $u^k$ solves the eigenproblem \eqref{eq:EV}.
\end{proposition}

\begin{proof}
First, \eqref{eq:EV} means that $cos\theta=\frac{\langle u^k,T(u^k)\rangle)}{\|u^k\|\|T(u^k)\|}=\frac{\langle u^k,\lambda u^k\rangle)}{\|u^k\|\|\lambda u^k\|}=\pm 1$.

Second, suppose that the angle between $u^k$ and $T(u^k)$ is $\pi n$. Then $\pm 1=cos \theta=\frac{\langle u^k,T(u^k)\rangle}{\|u^k\|\|T(u^k)\|} \underset{Lemma \; \ref{lemma}} \implies$ $\langle u^k,T(u^k)\rangle=\pm \|T(u^k)\| \implies$ $|\langle u^k,T(u^k)\rangle|= \|T(u^k)\|$. From Proposition \ref{prop3}, this only holds if $u^k$ and $T(u^k)$ are co-linear. Thus, \eqref{eq:EV} holds with $\lambda=\pm \|T(u^k)\|$.
\end{proof}

\subsection{Contraction Operators}
So far we have discussed the steady state behavior of the power iteration for a very generic $T(u)$, making very few assumptions regrading its nature. We will now follow previous work (e.g.~\cite{daskalakis2018converse},
~\cite{wu2017markov},~\cite{chu2014second},~\cite{kolda2011shifted}) treating the power iteration process as a fixed-point process (note, that applying the fixed-point iteration for functionals has been investigated before~\cite{browder1966solution}). We will also make the reasonable assumption that $T(u)$ is a \textit{contraction operator} (defined hereafter). Under this assumption we prove the \textit{convergence} of this process.
In Sec. \ref{Sec::res} we test if our nonlinear operators indeed behave as contraction operators. We will show that while the stronger condition does not hold, a weaker but sufficient condition does hold, and thus this assumption is valid.

We write the power iteration process $\{u^k\}$ as the fixed-point process $\{g^k\}$:

\begin{algorithm}{Power Iteration as a Fixed-point Process.}
\label{Algfix}
\begin{enumerate}
    \item Initialization: $g^0 \gets f^0$, $\|f^0\|=1$, $\,\,k \gets 1.$
    \item Repeat until $k=K$ or $\|g^{k+1}-g^k\|< \varepsilon$:\\ 
    $g(u^k) \equiv u^{k+1} \gets \sign(\langle u^k,T(u^k)\rangle)\frac{T(u^k)}{\|T(u^k)\|}$, $\,\,\,k \gets k+1.$
\end{enumerate}
\end{algorithm}

\begin{proposition}\label{prop_Lipschitz}
The fixed-point iteration process $\{g^k\}$ converges, that is, $\lim_{k\to\infty} \|u^k-u^{k-1}\|=0$, if $g$ is a contraction operator, that is, if $\forall x,y,$ there exists $L<1$, such that $g$ admits the following Lipschitz continuity property:
$ \|g(x)-g(y)\| \leq L\|x-y\| $.
\end{proposition}

\begin{proof}
This follows the Banach fixed point theorem. Following the Lipschitz continuity assumption, $\exists L<1$ such that the following statements hold:

$\|u^2-u^1\| \equiv \|g(u^1)-g(u^0)\| \leq L\|u^1-u^0\|$

$\|u^3-u^2\| \equiv \|g(u^2)-g(u^1)\| \leq L\|u^2-u^1\|\leq L^2\|u^1-u^0\|$ \; ...

$\|u^k-u^{k-1}\| \equiv \|g(u^{k-1})-g(u^{k-2})\| \leq L\|u^{k-1}-u^{k-2}\| \leq L^{k-1}\|u^1-u^0\|$

%
Now, for $L<1,$ 
$\lim_{k\to\infty} \|u^k-u^{k-1}\| \leq \lim_{k\to\infty} L^{k-1} \|u^1-u^0\|=0$, hence $\lim_{k\to\infty} \|u^k-u^{k-1}\|=0$.  
\end{proof}



\subsection{Induced Operators}
We present a useful concept of operators induced by a given \textbf{generic, non-linear} operator, which allow generating eigenfunctions with different characteristics than those of the given operator.
Most importantly, the complementary operator defined here easily allows generating eigenfunctions corresponding to the \textit{small} eigenvalue of the given operator.   

\begin{definition}{Complementary operator.}
The complementary operator corresponding to the generic operator $T(u)$ is defined as: $T^\dagger(u) \equiv u-T(u)$, such that $T+T^\dagger=I$, where $I$ is the identity operator.
\end{definition}
\begin{property}\label{prop_ortho_op}
If $u$ is an eigenfunction of $T(u)$ with an eigenvalue $\lambda$, then $u$ is also an eigenfunction of $T^\dagger(u)$ with an eigenvalue $(1-\lambda)$.
\end{property}
Since $\{u,\lambda\}$ are an eigenpair of $T(u)$ we have $T(u)=\lambda u$. Thus: $T^\dagger(u)\equiv u-T(u)=u-\lambda u=(1-\lambda)u$. Therefore, $\{u,(1-\lambda)\}$ are an eigenpair of $T^\dagger(u)$.
From Property \ref{prop_ortho_op}, the following useful property immediately follows:
\begin{property}\label{prop_large_small}
The eigenfunction corresponding to the large eigenvalue of $T^\dagger(u)$ is the eigenfunction corresponding to the small eigenvalue of $T(u)$.
\end{property}
A useful algorithmic property thus results. Since Algorithm \ref{AlgPower} generates an eigenfunction with the maximal possible eigenvalue (as seen numerically), applying it using $T^\dagger$ will generate an eigenfunction of $T$ with the \textit{minimal} possible eigenvalue. This is a natural extension of the linear case.

Remark 1. For $T$ with $\lambda \in [0,1]$ (typical for coarsening (denoising) operators), we get $\lambda \in [0,1]$ also for $T^\dagger$.

Remark 2. For $[0,\lambda_{max}]$ with $\lambda_{max}>1$, we can attain positive eigenvalues for $T^\dagger$ using the following variant: $T^\dagger := I - \alpha T$, where $\alpha \le \frac{1}{\lambda_{max}}$. 


We focus on denoising operators, designed to remove noise or simplify the image by removing fine-scale details, thus refer to $T^{\dagger}$ as a \emph{texture generator}.

We can also define an \emph{enhancing operator} by adding the textural part, weighted by $\alpha$:
    $T^E_\alpha(u) \equiv u+\alpha T^{\dagger}(u)=u(1+\alpha)-\alpha T(u), \textrm{   where  } \alpha>0.$
In this case, if $u$ is an eigenfunction of $T(u)$ with an eigenvalue $\lambda$, then $u$ is also an eigenfunction of $T^E_\alpha(u)$, with an eigenvalue $(1+\alpha-\alpha\lambda) \ge 1$.

\subsection{Generating More Eigenfunctions}\label{sec::more_ef}
As mentioned before, Algorithm \ref{AlgPower} generates a single eigenfunction (given $f^0$) with $\lambda$ very close to 1.
Additionally, as discussed following Property \ref{prop_large_small}, applying Algorithm \ref{AlgPower} with $T^\dagger$ yields an eigenfunction of $T$ with a small eigenvalue. We now address obtaining additional eigenfunctions stemming from $f^0$, following projections methods in the linear case.

\begin{definition}

Let $\{v_i\}^N_{i=1}$ be an orthonormal set of eigenfunctions of $T$, that is, each $v_i$ admits \eqref{eq:EV} and in addition, $\forall i,j, \langle v_i,v_j \rangle = \delta_{ij}$.
Then a \textbf{single projection} of $f$ onto the space \textbf{orthogonal} to $\{v_i\}^N_{i=1}$ is defined as: 
$f_N=f-\sum_{i=1}^N \langle f,v_i \rangle v_i$.
\end{definition}

\begin{property}
The single projection $f_N$ is orthogonal to the orthonormal set $\{v_i\}^N_{i=1}$: $\forall 1 \leq j \leq N, \langle v_j, f_N \rangle = 0$.
\end{property}
\begin{proof}
$\langle v_j, f_N \rangle = \langle v_j, f \rangle-\sum_{i=1}^N \langle f,v_i \rangle \langle v_i,v_j \rangle \underset{Orthonormality} = \langle v_j, f \rangle-\sum_{i=1}^N \langle f,v_i \rangle \delta_{ij}$ 
$\underset{\sum \neq 0 \; only \; for \; i=j} = \langle v_j, f \rangle-\langle f,v_j \rangle\ =0$.
\end{proof}


Now, to attain $v_{N+1}$, we initialize the process with $f^N$, a single projection of $f^0$. We then iteratively apply $T$ and perform a single projection:

\begin{algorithm}{Generating More Eigenfunctions.}
\label{AlgMore}
\begin{enumerate}
    \item Initialization: $f^N$, $\,\,k \gets 1.$
    \item Repeat until $k=K$ or $\|u^{k+1}-u^k\|< \varepsilon$:\\ 
    $z^k \gets T(u^k)$;\\
    $y^k \gets $ single projection of $z^k$;\\
    $u^{k+1} \gets \frac{y_k}{\|y_k\|}$, $\,\,\,k \gets k+1.$
\end{enumerate}
\end{algorithm}

We note that existing eigenfunctions may be non-orthogonal. In this case, projection should be done iteratively: $f_1 = \frac{\langle f,v_1 \rangle v_1}{\|v_1\|^2}$,
$f_2 = \frac{\langle f_1,v_2 \rangle v_2}{\|v_2\|^2}$,
and so on. 

It can be easily seen that when the process reaches an eigenfunction, it converges, assuming orthogonality of previous eigenfunctions, and the result is also orthogonal to previous ones.
In practice, the \textit{function} generated is orthogonal to the set, though we cannot guarantee this. However, it may be considered only as a \textit{pseudo}-eigenfunction, as it may not hold $T(u)=\lambda u$. Nevertheless, we can generate an eigenfunction admitting \eqref{eq:EV} by applying more power iterations without projections, at the expense of orthogonality. 
This agrees with recent TV and one-homogeneous functionals~\cite{benning2013,burger2016spectral,bungert2019nonlinear} theory, where eigenfunctions are not necessarily orthogonal.

\off{
All resulting eigenfunctions $\{v_i\}^{N+1}_{i=1}$ are orthonormal (shown empirically): 
\begin{property}\label{prop_ortho_eigs}
$\forall 1 \leq i,j \leq N+1, \langle v_i,v_j \rangle = \delta_{ij}\|v_j\|^2$.
\end{property}
}

\section{Experimental Results and Applications}\label{Sec::res}
In this section, we present experimental results for generating eigenfunctions of different generic, non-linear operators, both denoisers and, apparently for the first time, non-convex deblurring operators. We first validate our method using the well-established TV denoiser. We then show results for the black-box generic non-linear EPLL denoiser, and suggest an application of a decryption-encryption scheme. For both denoisers, we also show eigenfunction degradation robustness and decay profiles. Last, we demonstrate our method for the TV and EPLL deblurring operators.

We denote the degraded image, the restored image, the eigenvalue and the blurring kernel as $f$, $u$, $\lambda$ and $A$, respectively. 
When formulating optimization problems, $\eta>0$ is a fixed weight between the fidelity and prior terms. $x$ is a pixel in image domain $\Omega$. We empirically determine the number of iterations, such that the result admits the eigenfunction validation measures.

\subsection{Validation: TV Denoising Operator}
We show results for the well-established non-linear TV denoiser~\cite{rudin1992nonlinear}, which can be formulated as the following optimization problem:
\begin{equation}\label{Eq::TV}
    \underset{u}{\text{min}} \ \frac{\eta}{2} \|f-u\|^2 + \int_{\Omega} |\nabla u(x)| dx,
\end{equation}
This section serves two purposes. Applying our method to a well-studied, analytic operator serves as a proof of concept for our method, as we compare our results to known ones. We also present first results on the TV texture generator (eigenfunctions with small eigenvalues). 

To apply our method to the TV denoiser, we first test (Fig. \ref{Fig::contraction_TV}) whether TV behaves as a contraction operator. While the stronger Lipschitz condition $L_k<1$ ($L_k=\frac{\|u^k-u^{k-1}\|}{\|u^{k-1}-u^{k-2}\|}$) does not hold numerically $\forall k$, a weaker but sufficient condition holds: $\lim_{k\to\infty} \prod_{i=1}^{k} L_i=0$. Thus, the process converges, yielding an eigenfunction (Propositions \ref{prop_Lipschitz}, \ref{prop1}).

\begin{figure}
\captionsetup[subfigure]{justification=centering}
\centering
\begin{subfigure}{0.48\textwidth}
\centering
    \includegraphics[width=0.6\textwidth]{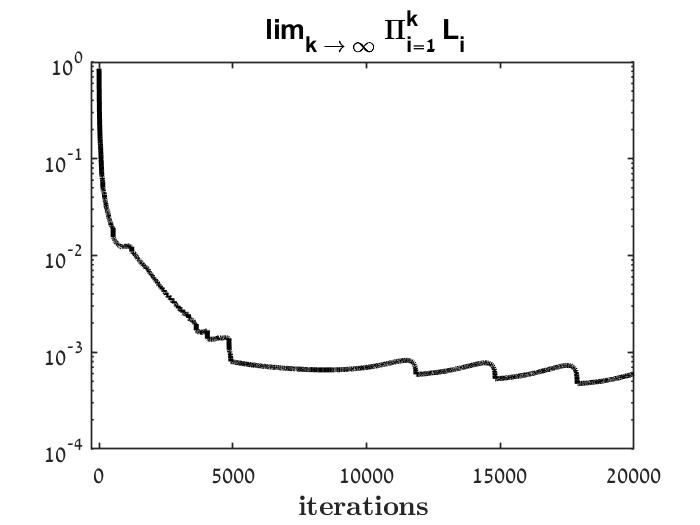}
    \caption{Weak condition holds for TV}
    \label{Fig::contraction_TV}
\end{subfigure}
\begin{subfigure}{0.48\textwidth}
\centering
    \includegraphics[width=0.6\textwidth]{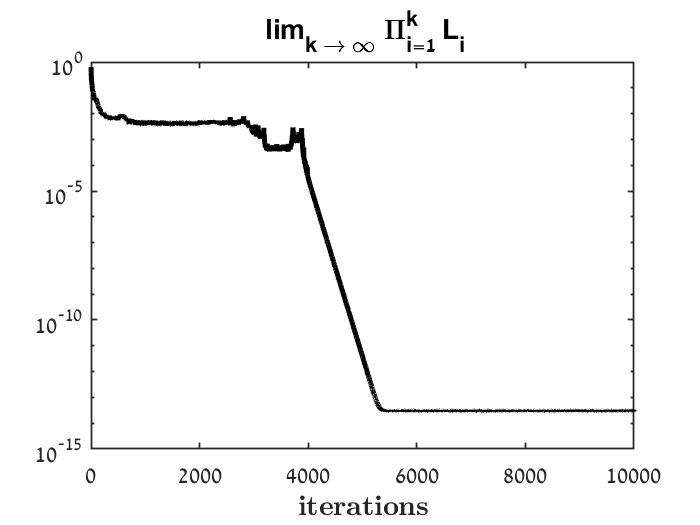}
    \caption{Weak condition holds for EPLL}
    \label{Fig::contraction_EPLL}
\end{subfigure}
\caption{TV and EPLL hold a weak condition for being contraction operators: $\lim_{k\to\infty} \prod_{i=1}^{k} L_i=0$ ($L_k=\frac{\|u^k-u^{k-1}\|}{\|u^{k-1}-u^{k-2}\|}$), thus the process converges (Proposition \ref{prop_Lipschitz}).}
\label{Fig::contraction}
\end{figure}

Fig. \ref{Fig::TV} 
shows the power iteration evolution of an initial image to the final eigenfunction. Note that the eigenvalue is smaller than but very close to 1, as expected from a detail-attenuating operator, and in accordance with Corollary \ref{cor_lam}. The eigenfunction represents the coarse structure of the initial image, and its shape is in accordance with the convex nature of TV eigenfunctions~\cite{gilboa2018beyond}. We also validate (Sec. \ref{Sec::measures}) that this is an eigenfunction. Note that the specific method of discretization can affect the eigenfunction structure (we use~\cite{chambolle2004}). 
Fig. \ref{Fig::decay_TV1}-\ref{Fig::decay_TV2} 
shows the eigenfunction decay when the denoiser is iteratively applied: for $98\%$ of pixels, decay profiles exhibit a distinct pattern, consistent with the theory of TV eigenfunctions~\cite{burger2016spectral} as analyzed in the context of decay profiles in~\cite{katzir}.

\begin{figure*}
\captionsetup[subfigure]{justification=centering}
\centering
\begin{subfigure}{0.155\textwidth}
\centering
    \includegraphics[width=0.9\textwidth]{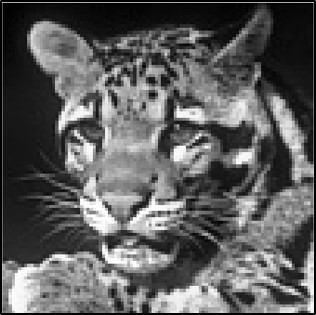}
    \caption{Initial image}
\end{subfigure}
\begin{subfigure}{0.155\textwidth}
\centering
    \includegraphics[width=0.9\textwidth]{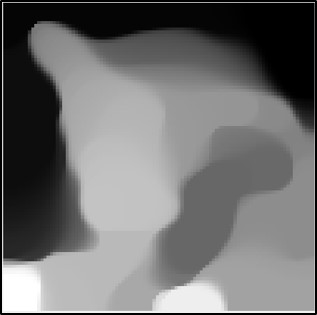}
    \caption{100 iterations}
\end{subfigure}
\begin{subfigure}{0.155\textwidth}
\centering
    \includegraphics[width=0.9\textwidth]{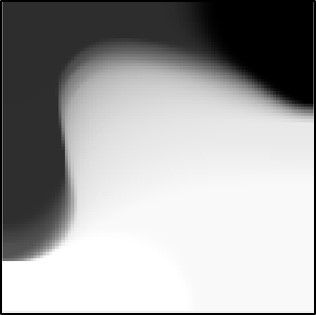}
    \caption{500 iterations}
\end{subfigure}
\begin{subfigure}{0.155\textwidth}
\centering
    \includegraphics[width=0.9\textwidth]{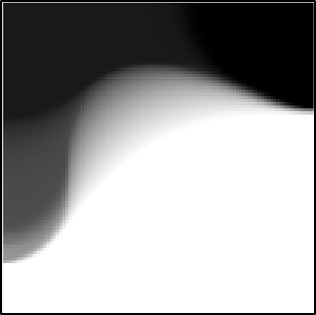}
    \caption{1000 iterations}
\end{subfigure}
\begin{subfigure}{0.155\textwidth}
\centering
    \includegraphics[width=0.9\textwidth]{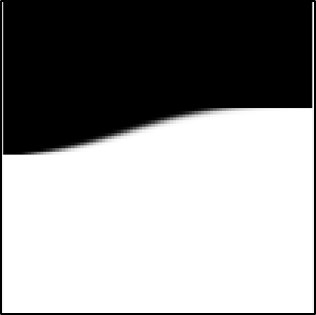}
    \caption{5000 iterations}
\end{subfigure}
\begin{subfigure}{0.155\textwidth}
\centering
    \includegraphics[width=0.9\textwidth]{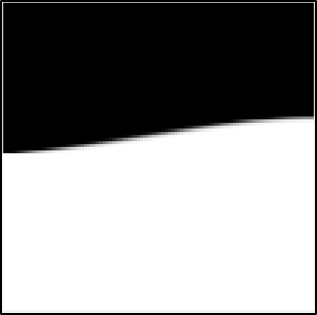}
    \caption{eigenfunction}
\end{subfigure}
\begin{subfigure}{0.24\textwidth}
\centering
    \includegraphics[height=2cm]{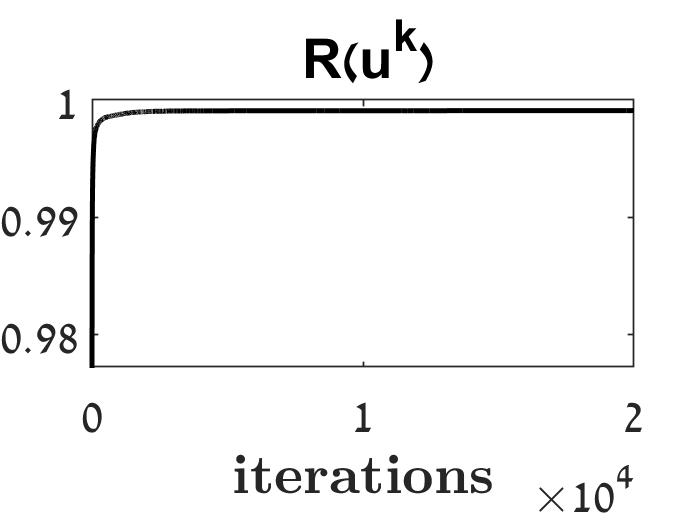}
    \caption{Rayleigh quotient $\rightarrow \lambda$}
\end{subfigure}
\begin{subfigure}{0.24\textwidth}
\centering
    \includegraphics[height=2cm]{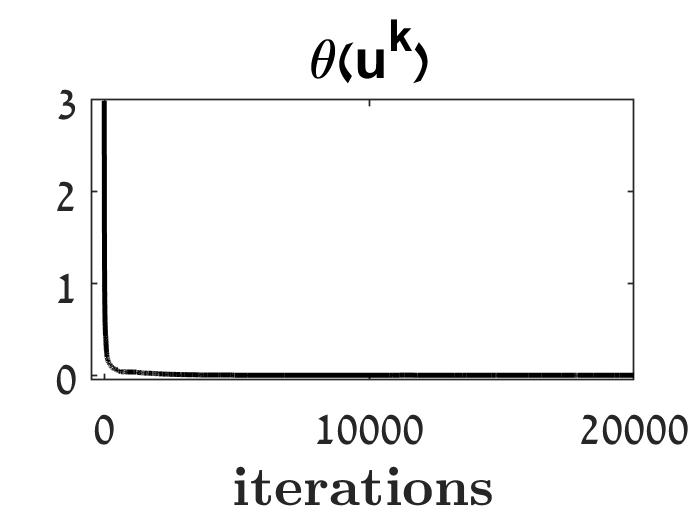}
    \caption{$\theta$ between $u^k, T(u^k) \rightarrow 0^{\circ}$}
\end{subfigure}
\begin{subfigure}{0.25\textwidth}
\centering
    \includegraphics[height=2cm]{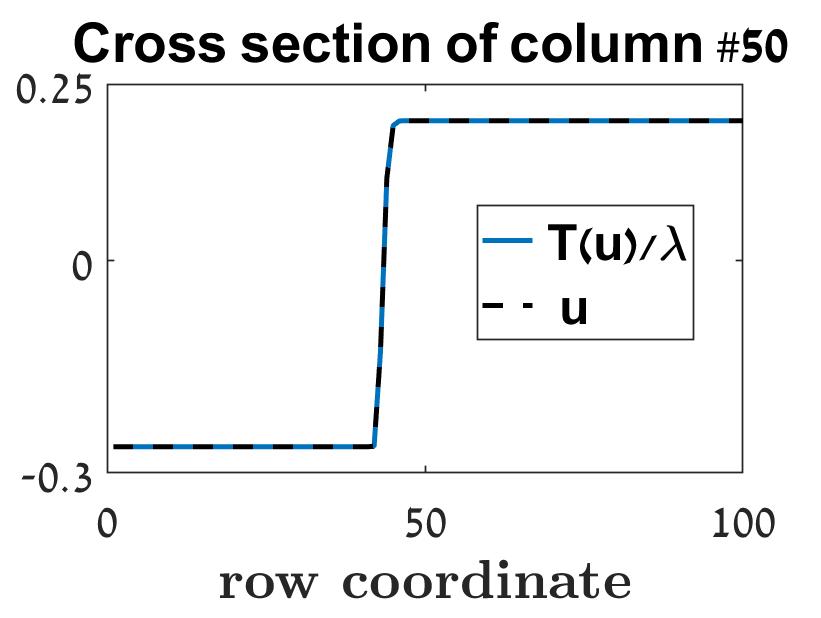}
    \caption{$\frac{T(u)}{\lambda}, u$: Identical cross sections}
\end{subfigure}
\begin{subfigure}{0.24\textwidth}
\centering
    \includegraphics[height=2cm]{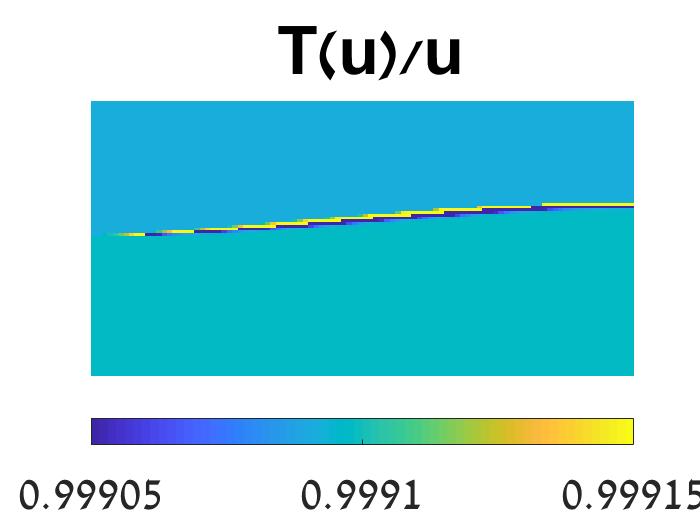}
    \caption{Constant Values of $\frac{T(u)}{u}$}
\end{subfigure}
\caption{TV power iteration evolution to final eigenfunction $u$ (20,000 iterations) with $\lambda=0.9991$. We validate this is an eigenfunction: the Rayleigh quotient increases to the eigenvalue (in accordance with Propositions \ref{prop3}, \ref{prop5}), and the angle between $u^k, T(u^k)$ decreases to zero (in accordance with Proposition \ref{prop4}). Cross sections of $\frac{T(u)}{\lambda}, u$ are identical, and $\frac{T(u)}{u}=C$ ($98\%$ of values are in the displayed range).}
\label{Fig::TV}
\end{figure*}

\begin{figure*}
\captionsetup[subfigure]{justification=centering}
\centering
\begin{subfigure}{0.18\textwidth}
\centering
    \includegraphics[height=2.1cm]{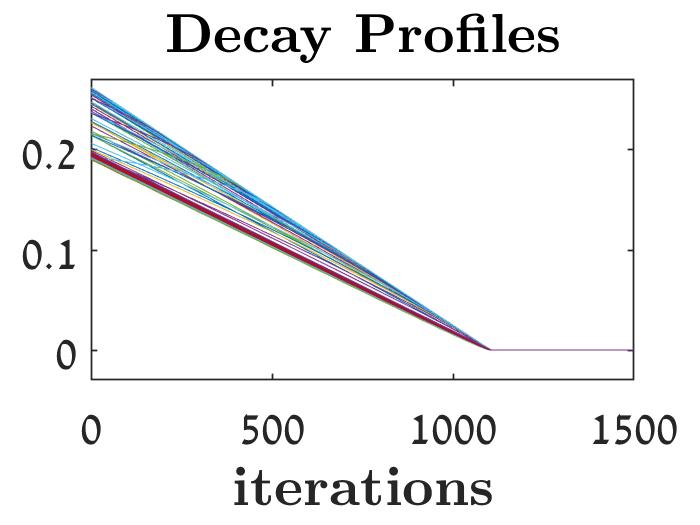}
    \caption{\\TV, unnormalized}
    \label{Fig::decay_TV1}
\end{subfigure}
\begin{subfigure}{0.18\textwidth}
\centering
    \includegraphics[height=2.1cm]{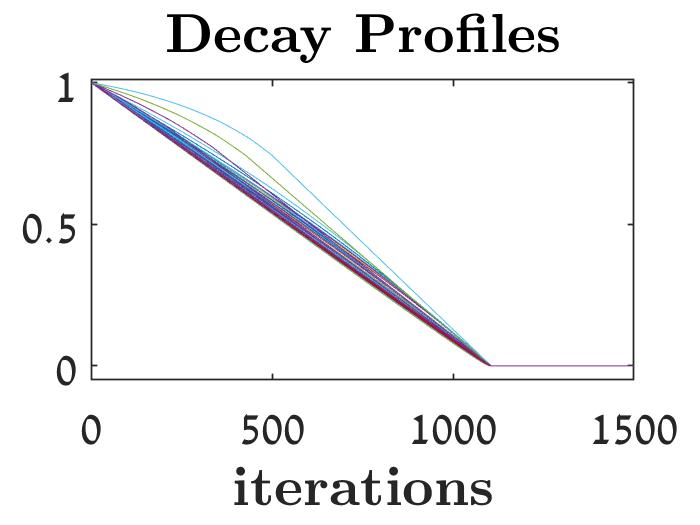}
    \caption{\\TV, normalized}
        \label{Fig::decay_TV2}
\end{subfigure}
\begin{subfigure}{0.18\textwidth}
\centering
    \includegraphics[height=2.1cm]{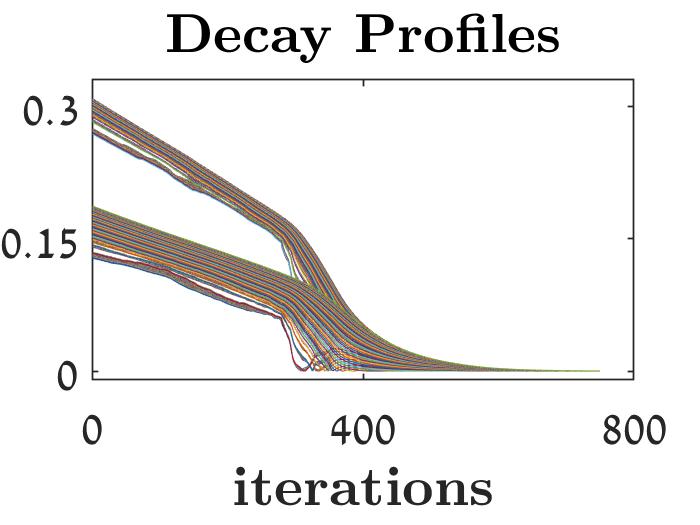}
    \caption{\\EPLL, unnormalized}
     \label{Fig::decay_EPLL1}
\end{subfigure}
\begin{subfigure}{0.18\textwidth}
\centering
    \includegraphics[height=2.1cm]{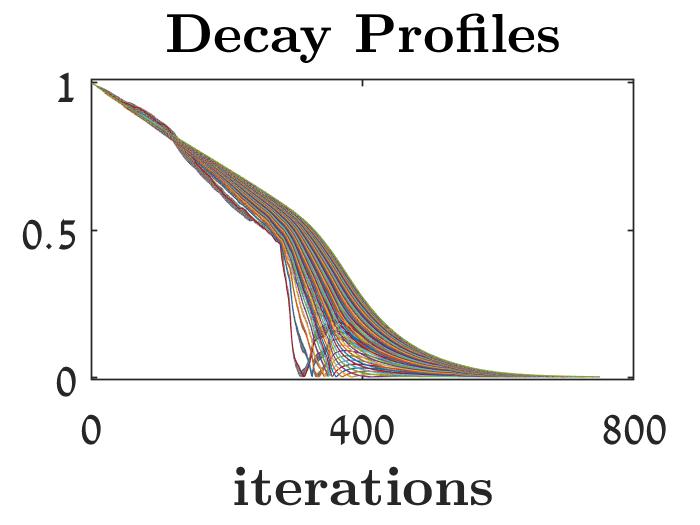}
    \caption{\\EPLL, normalized}
    \label{Fig::decay_EPLL2}
\end{subfigure}
\begin{subfigure}{0.23\textwidth}
\centering
    \includegraphics[height=2.1cm]{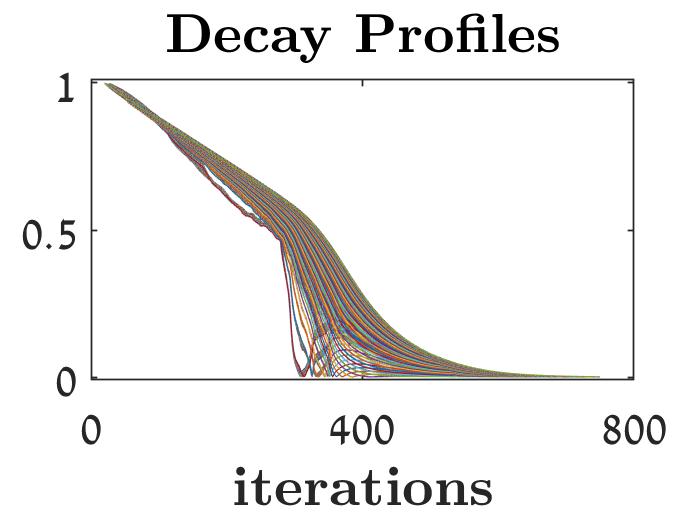}
    \caption{EPLL under noise: truncated and normalized}
    \label{Fig::decay_noise}
\end{subfigure}
\caption{Decay profiles (per pixel) of eigenfunctions, when the corresponding denoiser is iteratively applied. (a)-(b): TV decay profiles show a distinct pattern, consistent with TV theory. (c)-(d): EPLL decay profiles also show a distinct pattern (though there is no EPLL theory for comparison). (e): We truncate the distorted beginning of EPLL decay profiles of a degraded eigenfunction (noise, $\sigma=0.1$), which results in similar profiles.}
\end{figure*}
Last, we present the induced TV texture generator and generate its eigenfunction (Fig. \ref{Fig::TV_noise_gen}), 
which is with an eigenvalue of 1. This is also the small eigenfunction of the TV denoiser with eigenvalue 0 (see Property \ref{prop_large_small}). Indeed TV denoising \eqref{Eq::TV} can yield a solution $u=0$, when TV of $u^k$ is high enough (depending on $\eta$)~\cite{Meyer} (see a generalization for arbitrary 1-homogeneous functionals in~\cite{bungert2019solution}). Indeed the final result represents the texture of the initial image. Again, we validate this is an eigenfunction.

\begin{figure*}
\captionsetup[subfigure]{justification=centering}
\centering
\begin{subfigure}{0.18\textwidth}
\centering
    \includegraphics[width=0.8\textwidth]{Figs/tiger.jpg}
    \caption{Initial image}
\end{subfigure}
\begin{subfigure}{0.18\textwidth}
\centering
    \includegraphics[width=0.8\textwidth]{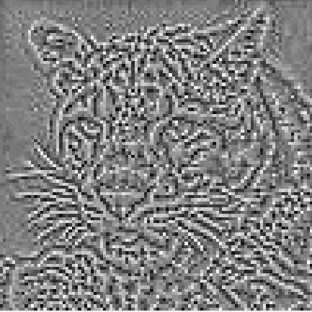}
    \caption{eigenfunction}
\end{subfigure}
\begin{subfigure}{0.2\textwidth}
\centering
    \includegraphics[height=2.1cm]{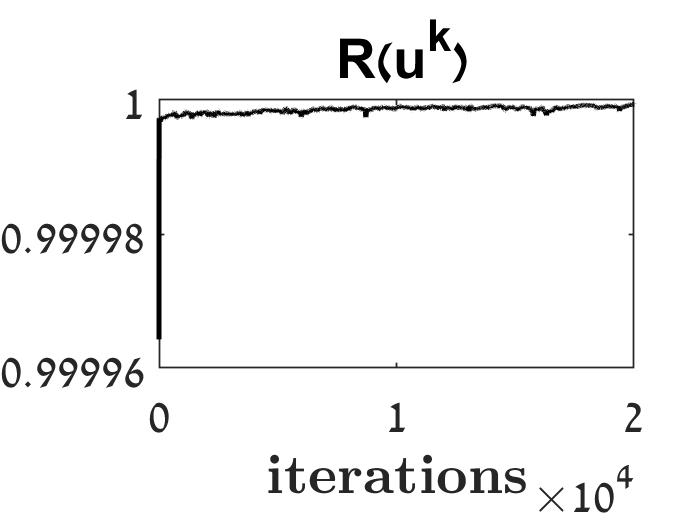}
    \caption{Rayleigh quotient $\rightarrow \lambda$}
\end{subfigure}
\begin{subfigure}{0.2\textwidth}
\centering
    \includegraphics[height=2.1cm]{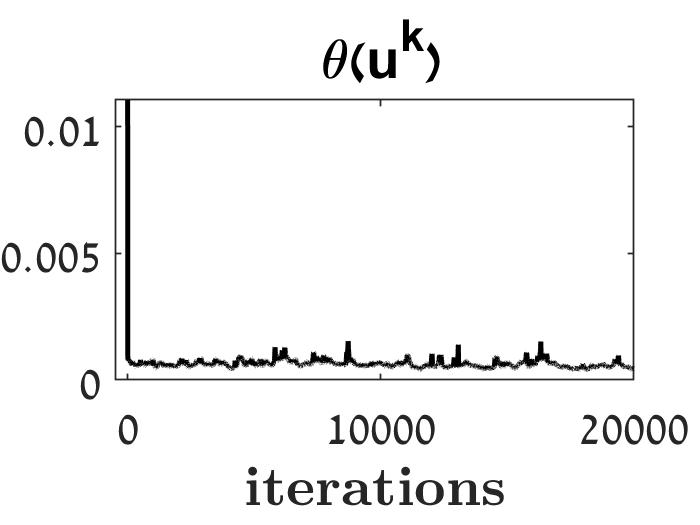}
    \caption{$\theta$ between $u^k, T(u^k) \rightarrow 0^{\circ}$}
\end{subfigure}
\begin{subfigure}{0.2\textwidth}
\centering
    \includegraphics[height=2.1cm]{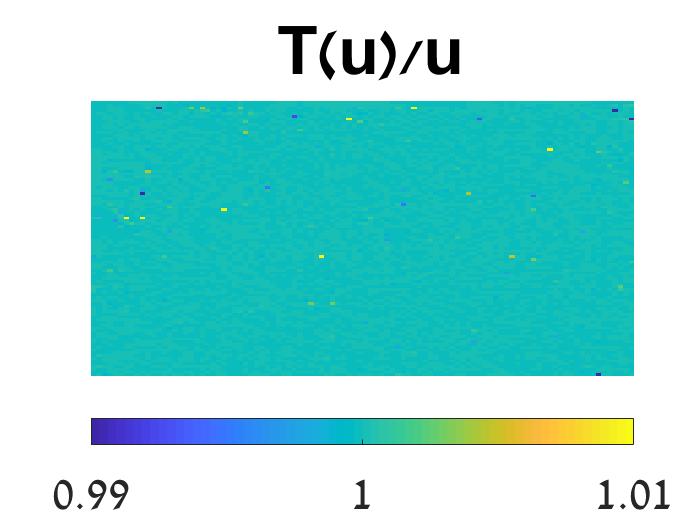}
    \caption{Constant Values of $\frac{T(u)}{u}$}
\end{subfigure}
\caption{TV texture generator eigenfunction $u$, with $\lambda=1$. We validate this is an eigenfunction following Propositions \ref{prop3}-\ref{prop4}. Also, $\frac{T(u)}{u}=C$ ($99.9\%$ of values are in the displayed range).}
\label{Fig::TV_noise_gen}
\end{figure*}

\subsection{EPLL Denoising Operator}\label{Sec::EPLL_res}
We show experimental results for the black-box generic EPLL (Expected Patch Log Likelihood) denoiser~\cite{zoran2011learning}. EPLL uses a generic framework for efficient image restoration, using a prior $p$ on image patches (selected using the operator $P_i$), formulated as the following optimization problem:
\begin{equation}\label{Eq::EPLL}
\begin{split}
    \underset{u}{\text{min}} \ \frac{\eta}{2}\|f-u\|^2 - EPLL_p(u), \\ EPLL_p(u) = \sum_{i} log \; p(P_i u)
\end{split}
\end{equation}
To increase performance, the paper suggests using the simple Gaussian Mixture Model (GMM) prior, learned from a set of natural images: $log \; p(x) = log \; \Bigg(\sum_{k=1}^{K} \pi_k \; N(x \;| \;\mu_k, \; \Sigma_k) \Bigg)$, where $K$, $\pi_k$, $\mu_k$, $\Sigma_k$ are the number of Gaussian mixtures, their mixing weights, their means and their covariance matrices, respectively. EPLL was shown to prefer large structures, straight borders and round corners~\cite{shaham2016visualizing}.

We first apply our method to the EPLL denoiser itself to generate a large eigenfunction. Note that we drop the output clipping stage within EPLL to avoid evolution in the linear region of the operator.
We first test (Fig. \ref{Fig::contraction_EPLL}) whether EPLL behaves as a contraction operator. Again, a weaker but sufficient Lipschitz condition holds: $\lim_{k\to\infty} \prod_{i=1}^{k} L_i=0$. Thus, the process converges, and converges into an eigenfunction (Propositions \ref{prop_Lipschitz}, \ref{prop1}).

Fig. \ref{Fig::EPLL}
shows the power iteration evolution of two different initial images to the final eigenfunctions. Note that the eigenvalues are smaller than but very close to 1. This is as expected from a detail-attenuating operator, and is in accordance with Corollary \ref{cor_lam}. Eigenfunctions are different, as each represents the coarse structure of a different initial image, and their shapes are in accordance with the observations made in~\cite{shaham2016visualizing}. We also validate these are eigenfunctions, where we see a similar behavior for both eigenpairs. Fig. \ref{Fig::better_removal}
compares stable and unstable modes of the denoiser to natural images. It shows a known eigenfunction behavior: using the corresponding denoiser, noise is better removed from the large eigenfunctions (stable modes), than from natural images, and from small eigenfunctions (unstable modes). 

\begin{figure*}[t!]
\captionsetup[subfigure]{justification=centering}
\centering
\begin{subfigure}{0.155\textwidth}
\centering
    \includegraphics[width=0.9\textwidth]{Figs/tiger.jpg}
    \caption{Initial image}
\end{subfigure}
\begin{subfigure}{0.155\textwidth}
\centering
    \includegraphics[width=0.9\textwidth]{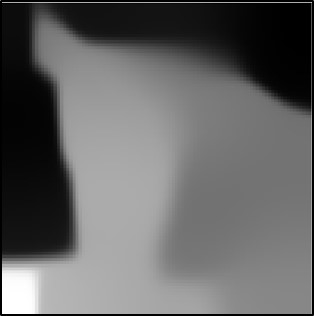}
    \caption{50 iterations}
\end{subfigure}
\begin{subfigure}{0.155\textwidth}
\centering
    \includegraphics[width=0.9\textwidth]{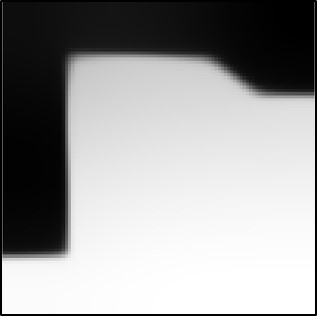}
    \caption{200 iterations}
\end{subfigure}
\begin{subfigure}{0.155\textwidth}
\centering
    \includegraphics[width=0.9\textwidth]{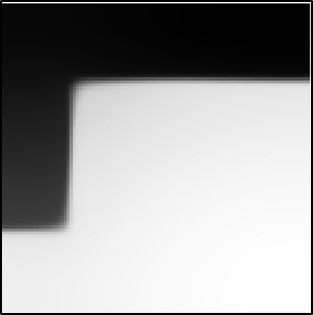}
    \caption{750 iterations}
\end{subfigure}
\begin{subfigure}{0.155\textwidth}
\centering
    \includegraphics[width=0.9\textwidth]{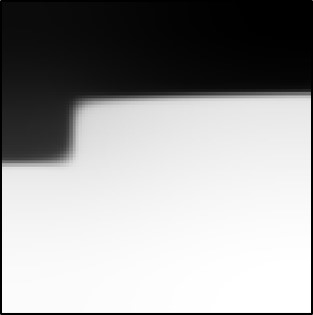}
    \caption{2000 iterations}
\end{subfigure}
\begin{subfigure}{0.155\textwidth}
\centering
    \includegraphics[width=0.9\textwidth]{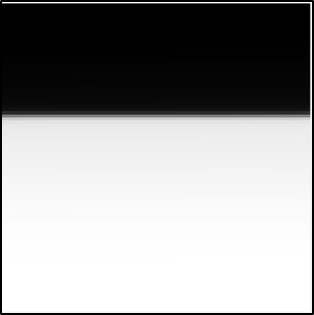}
    \caption{eigenfunction}
\end{subfigure}
\begin{subfigure}{0.24\textwidth}
\centering
    \includegraphics[height=2cm]{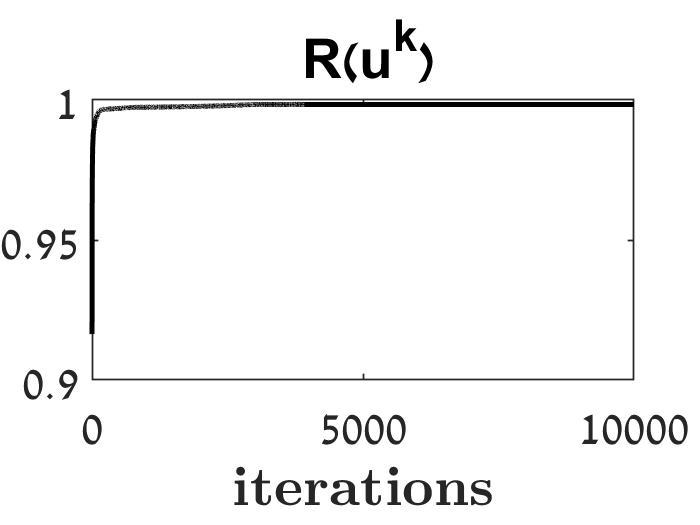}
    \caption{Rayleigh quotient $\rightarrow \lambda$}
\end{subfigure}
\begin{subfigure}{0.24\textwidth}
\centering
    \includegraphics[height=2cm]{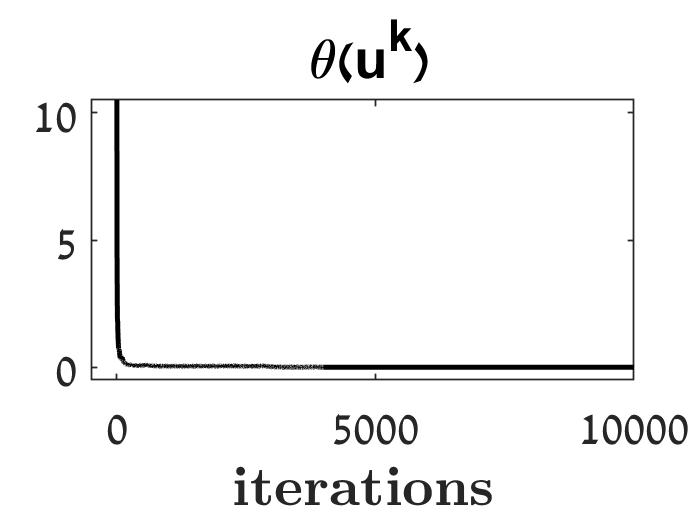}
    \caption{$\theta$ between $u^k, T(u^k) \rightarrow 0^{\circ}$}
\end{subfigure}
\begin{subfigure}{0.25\textwidth}
\centering
    \includegraphics[height=2cm]{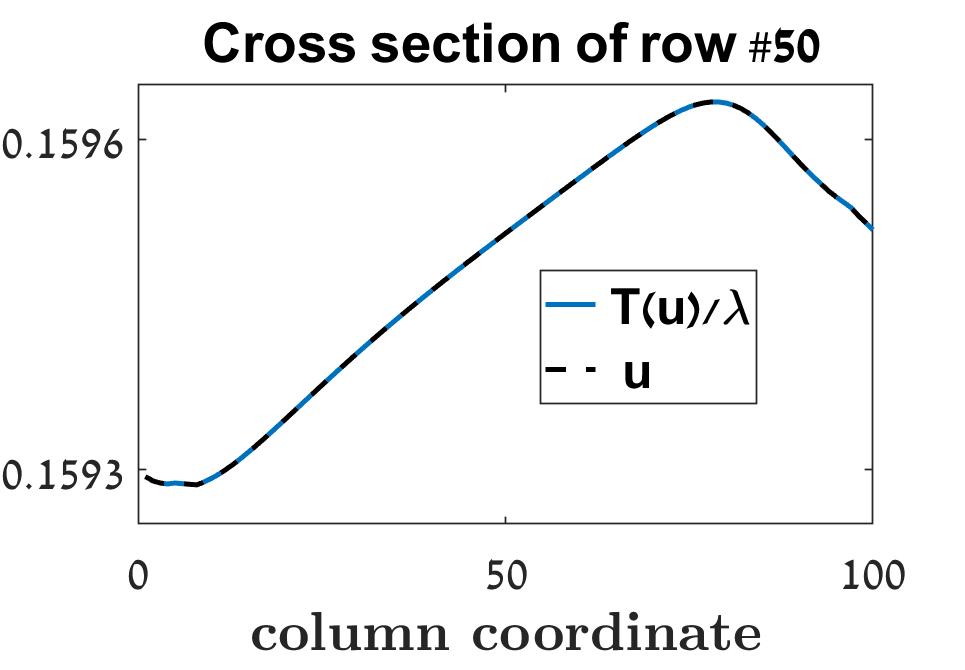}
    \caption{$\frac{T(u)}{\lambda}, u$: Identical cross sections}
\end{subfigure}
\begin{subfigure}{0.24\textwidth}
\centering
    \includegraphics[height=2cm]{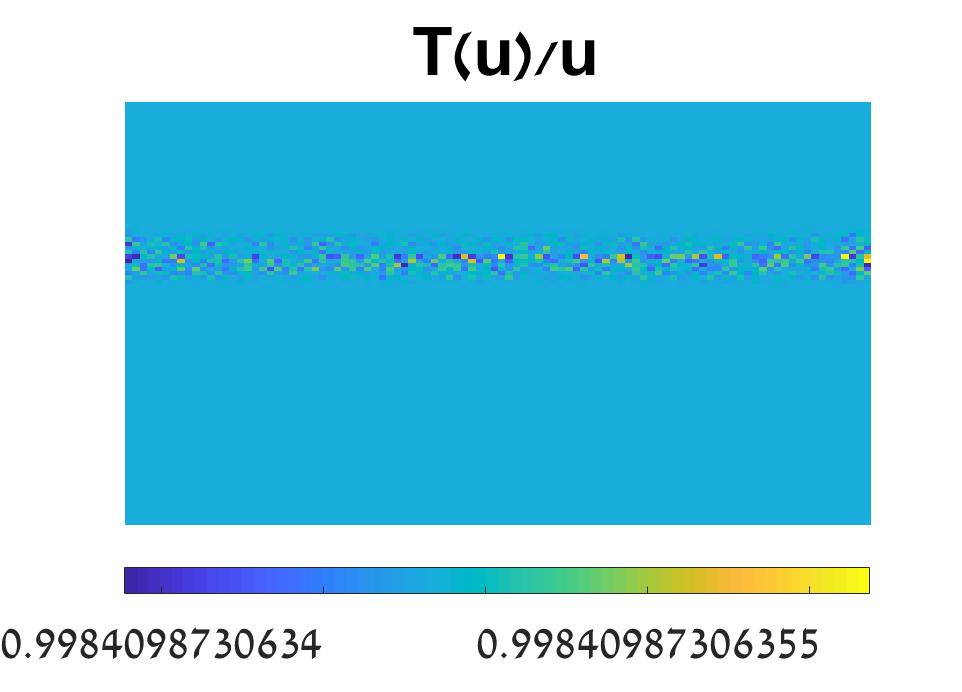}
    \caption{Constant Values of $\frac{T(u)}{u}$}
\end{subfigure}
\begin{subfigure}{0.18\textwidth}
\centering
    \includegraphics[width=0.7\textwidth]{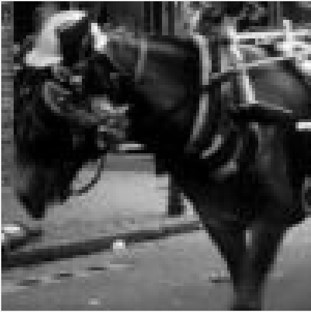}
    \caption{Initial image}
\end{subfigure}
\begin{subfigure}{0.18\textwidth}
\centering
    \includegraphics[width=0.7\textwidth]{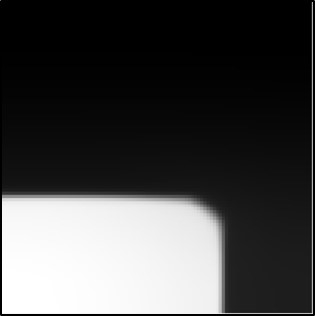}
    \caption{eigenfunction}
\end{subfigure}
\begin{subfigure}{0.2\textwidth}
\centering
    \includegraphics[height=2.1cm]{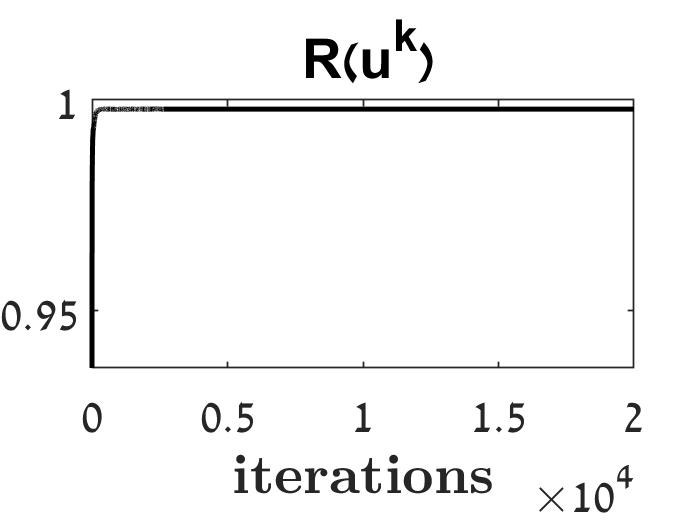}
    \caption{Rayleigh quotient $\rightarrow \lambda$}
\end{subfigure}
\begin{subfigure}{0.2\textwidth}
\centering
    \includegraphics[height=2.1cm]{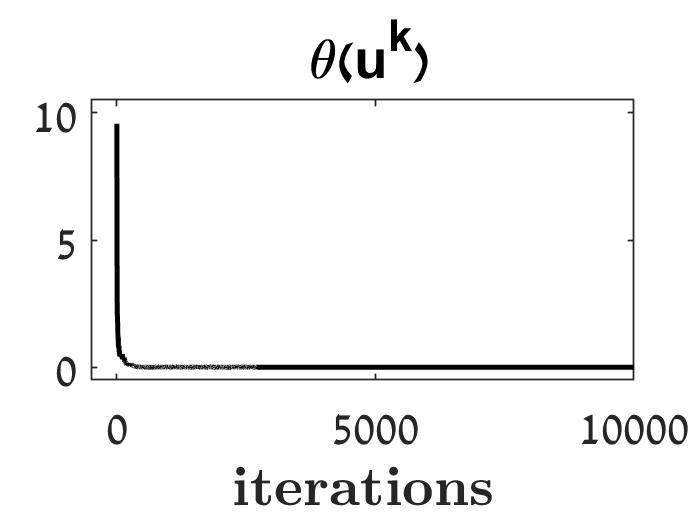}
    \caption{$\theta$ between $u^k, T(u^k) \rightarrow 0^{\circ}$}
\end{subfigure}
\begin{subfigure}{0.2\textwidth}
\centering
    \includegraphics[height=2.1cm]{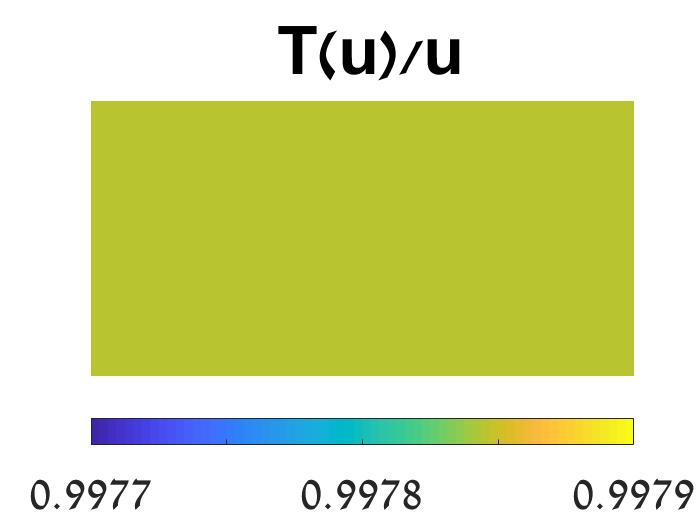}
    \caption{Constant Values of $\frac{T(u)}{u}$}
\end{subfigure}
\caption{EPLL power iteration evolutions to final eigenfunctions $u$. Top rows: 10,000 iterations of tiger image, $\lambda=0.9984$. Bottom row: 20,000 iterations of horse image, $\lambda=0.9978$. We validate these are eigenfunctions following Propositions \ref{prop3}-\ref{prop4}. Also, cross sections of $\frac{T(u)}{\lambda}, u$ are identical, and $\frac{T(u)}{u}=C$ (all values are in the displayed ranges).}
\label{Fig::EPLL}
\end{figure*}

\begin{figure*}
\captionsetup[subfigure]{justification=centering}
\centering
\begin{subfigure}{0.15\textwidth}
\centering
    \includegraphics[width=0.9\textwidth]{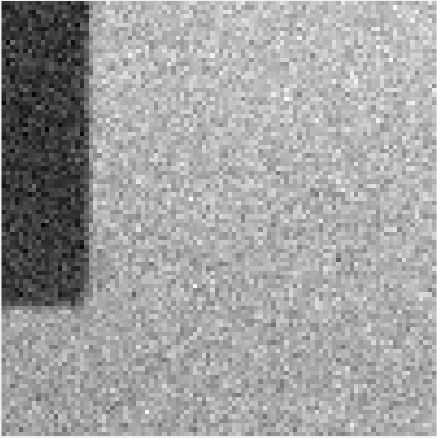}
    \caption{Noisy eigenfunction}
\end{subfigure}
\begin{subfigure}{0.15\textwidth}
\centering
    \includegraphics[width=0.9\textwidth]{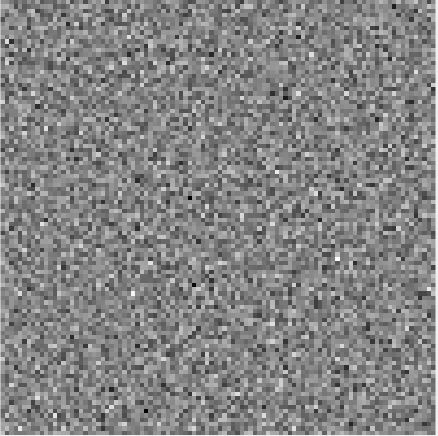}
    \caption{Noise removed from (a)}
\end{subfigure}
\begin{subfigure}{0.15\textwidth}
\centering
    \includegraphics[width=0.9\textwidth]{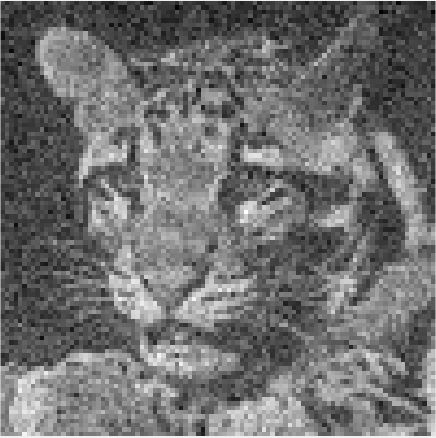}
    \caption{Noisy natural image}
\end{subfigure}
\begin{subfigure}{0.15\textwidth}
\centering
    \includegraphics[width=0.9\textwidth]{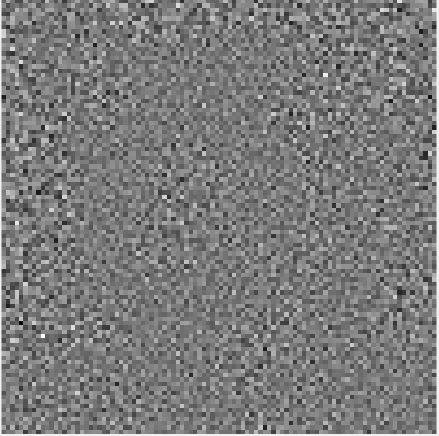}
    \caption{Noise removed from (c)}
\end{subfigure}
\begin{subfigure}{0.15\textwidth}
\centering
    \includegraphics[width=0.9\textwidth]{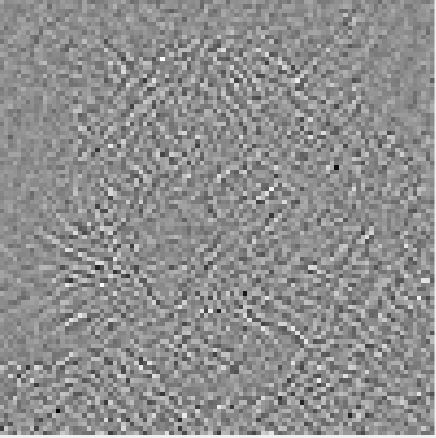}
    \caption{Noisy eigenfunction}
\end{subfigure}
\begin{subfigure}{0.15\textwidth}
\centering
    \includegraphics[width=0.9\textwidth]{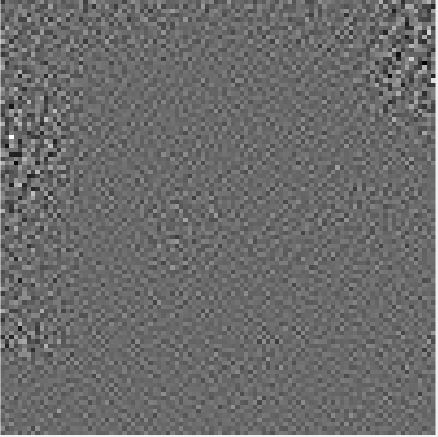}
    \caption{Noise removed from (e)}
\end{subfigure}
\\
\begin{subfigure}{0.45\textwidth}
\centering
    \includegraphics[width=0.8\textwidth]{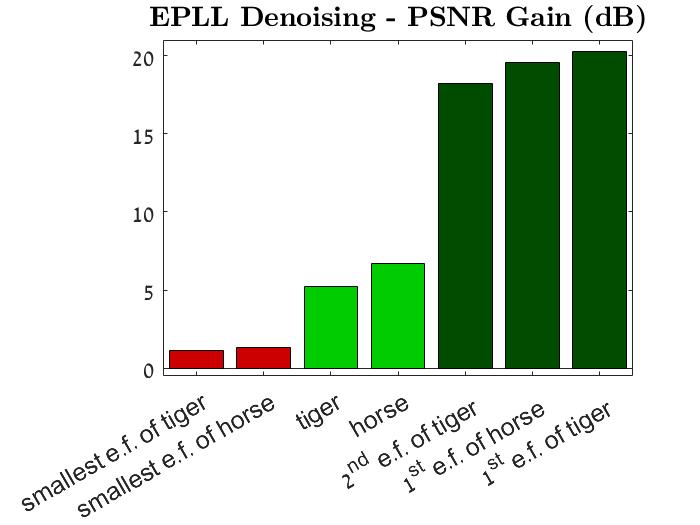}\caption{PSNR gain: eigenfunctions vs. natural images, $var_{noise}=\frac{1}{5}var_{img}$}
\end{subfigure}
\begin{subfigure}{0.45\textwidth}
\centering
\includegraphics[width=0.8\textwidth]{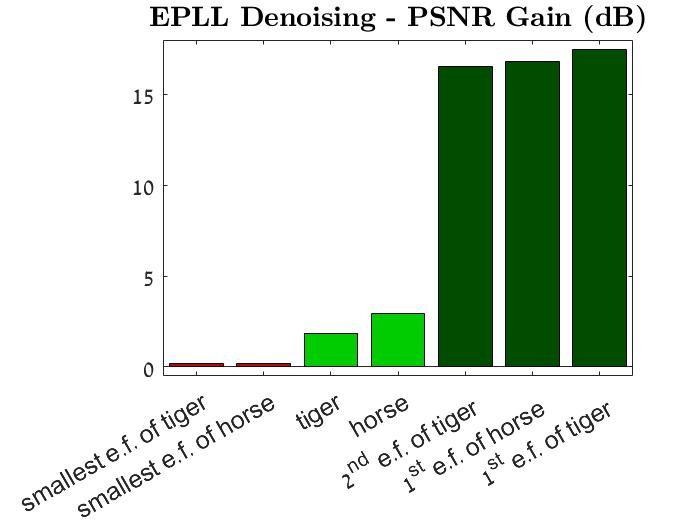}\caption{PSNR gain: eigenfunctions vs. natural images, $var_{noise}=\frac{1}{75}var_{img}$}
\end{subfigure}
\begin{subfigure}{0.133\textwidth}
\centering
    \includegraphics[width=0.8\textwidth]{Figs/EPLL_10000it.jpg}
    \caption{$1^{st}$ e.f. \\of tiger}
\end{subfigure}
\begin{subfigure}{0.133\textwidth}
\centering
    \includegraphics[width=0.8\textwidth]{Figs/EPLL_horse_20000it.jpg}
    \caption{$1^{st}$ e.f. of horse}
\end{subfigure}
\begin{subfigure}{0.133\textwidth}
\centering
    \includegraphics[width=0.8\textwidth]{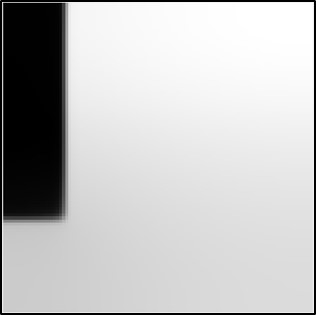}
    \caption{$2^{nd}$ e.f. of tiger}
\end{subfigure}
\begin{subfigure}{0.12\textwidth}
\centering
    \includegraphics[width=0.9\textwidth]{Figs/horse.jpg}
    \caption{\\horse}
\end{subfigure}
\begin{subfigure}{0.12\textwidth}
\centering
    \includegraphics[width=0.9\textwidth]{Figs/tiger.jpg}
    \caption{\\tiger}
\end{subfigure}
\begin{subfigure}{0.145\textwidth}
\centering
    \includegraphics[width=0.7\textwidth]{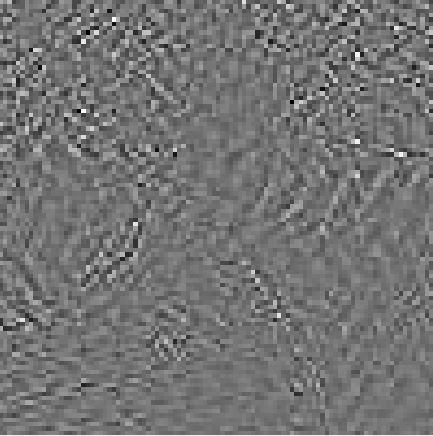}
    \caption{small e.f. of horse}
\end{subfigure}
\begin{subfigure}{0.16\textwidth}
\centering
    \includegraphics[width=0.65\textwidth]{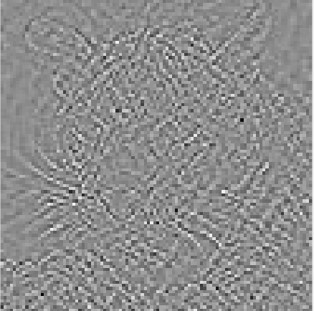}
    \caption{small e.f. \\of tiger}
\end{subfigure}
\caption{Demonstration of stable vs. unstable modes for EPLL. We show a known eigenfunction property: a denoiser better denoises its large eigenfunctions (stable modes), than natural images, than its small eigenfunctions (unstable modes). When noising with $var_{noise}=\frac{1}{5}var_{img}$ and denoising with EPLL, more noise is removed from the $2^{nd}$ eigenfunction (b), than from a natural image (d), than from the small eigenfunction (f), and it is more uniform. (g-h): PSNR gain when denoising using EPLL for different noise levels and different images (i-o): eigenfunctions vs. natural images.}
\label{Fig::better_removal}
\end{figure*}

Fig. \ref{Fig::decay_EPLL1}-\ref{Fig::decay_EPLL2} 
shows the distinct pattern of decay profiles for $97\%$ of pixels. Obviously,
as opposed to TV, EPLL has no decay profile theory to compare to. 
Following Sec. \ref{sec::more_ef}, we generate the second large eigenfunction (Fig. \ref{Fig::EPLL_2nd}), 
orthogonal to the large eigenfunction. However, the process keeps evolving and thus does not hold \eqref{eq:EV}. Thus, it can only be considered as a \textbf{pseudo}-eigenfunction (see~\cite{nossek2018flows}).
\off{However, we can generate a "processed" eigenfunction, which does hold $T(u)=\lambda u$, by applying more no-projection power iterations, at the expense of orthogonality. 
This agrees with the theory developed in recent years for TV and one-homogeneous functionals in general~\cite{benning2013,burger2016spectral,bungert2019nonlinear}, where the eigenfunctions are generally not orthogonal to each other.}
Finally, Fig. \ref{Fig::EPLL_noise_gen} 
shows the large eigenfunction of the EPLL texture generator, with an eigenvalue very close to 1. This is also a small eigenfunction of the EPLL denoiser, following Property \ref{prop_large_small} (it is easy to validate that Property \ref{prop_ortho_op} holds), and indeed, it represents the texture of the initial image. We also validate this is an eigenfunction. 
\begin{figure*}
\captionsetup[subfigure]{justification=centering}
\centering
\begin{subfigure}{0.18\textwidth}
\centering
\includegraphics[width=0.8\textwidth]{Figs/tiger.jpg}
\caption{Initial image}
\end{subfigure}
\begin{subfigure}{0.18\textwidth}
\centering
\includegraphics[width=0.8\textwidth]{Figs/EPLL_2nd_23000it.jpg}
\caption{Eigenfunction}
\end{subfigure}
\begin{subfigure}{0.2\textwidth}
\centering
\includegraphics[height=2.1cm]{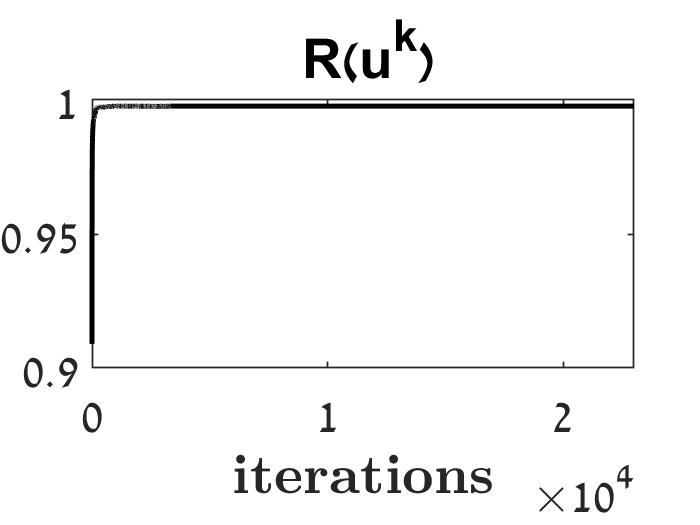}
\caption{Rayleigh quotient $\rightarrow \lambda$}
\end{subfigure}
\begin{subfigure}{0.2\textwidth}
\centering
\includegraphics[height=2.1cm]{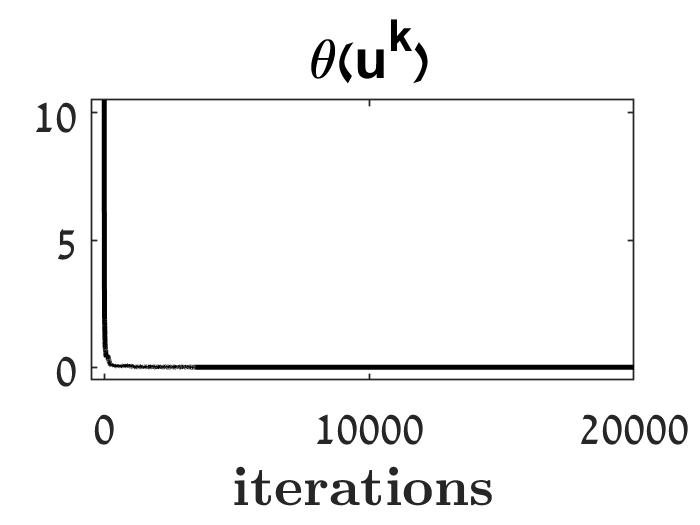}
\caption{$\theta$ between $u^k, T(u^k) \rightarrow 0^{\circ}$}
\end{subfigure}
\begin{subfigure}{0.2\textwidth}
\centering
\includegraphics[height=2.1cm]{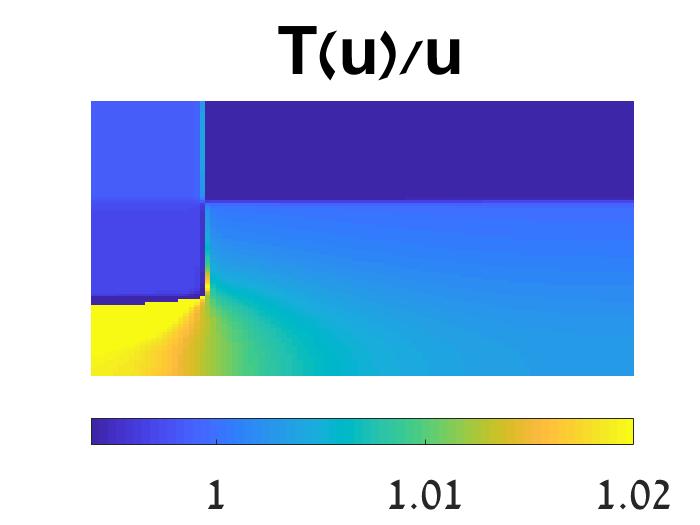}
\caption{\textbf{Inconstant} Values of $\frac{T(u)}{u}$}
\end{subfigure}
\caption{EPLL $2^{nd}$ large eigenfunction $u$, with $\lambda=0.9977$. We validate that Propositions \ref{prop3}-\ref{prop4} hold. However, as $\frac{T(u)}{u}=C$ does not hold, this is only a \textbf{pseudo}-eigenfunction.}
\label{Fig::EPLL_2nd}
\end{figure*}

\begin{figure*}
\captionsetup[subfigure]{justification=centering}
\centering
\begin{subfigure}{0.18\textwidth}
\centering
\includegraphics[width=0.8\textwidth]{Figs/tiger.jpg}
\caption{Initial image}
\end{subfigure}
\begin{subfigure}{0.18\textwidth}
\centering
\includegraphics[width=0.8\textwidth]{Figs/EPLL_noise_gen_20750it.jpg}
\caption{Eigenfunction}
\end{subfigure}
\begin{subfigure}{0.2\textwidth}
\centering
\includegraphics[height=2.1cm]{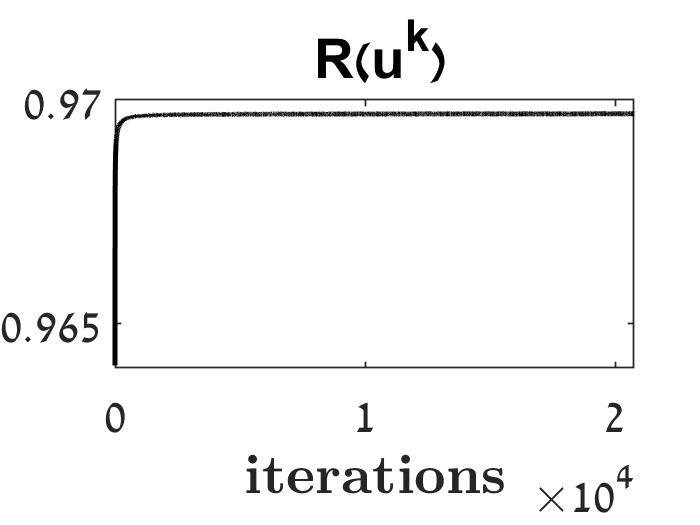}
\caption{Rayleigh quotient $\rightarrow \lambda$}
\end{subfigure}
\begin{subfigure}{0.2\textwidth}
\centering
\includegraphics[height=2.1cm]{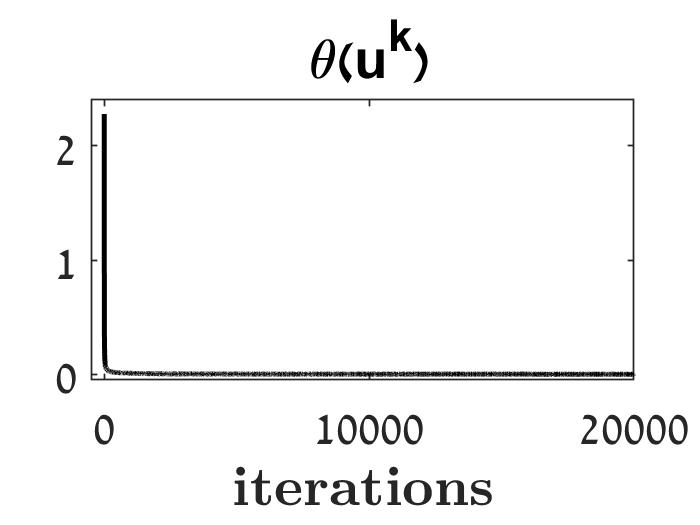}
\caption{$\theta$ between $u^k, T(u^k) \rightarrow 0^{\circ}$}
\end{subfigure}
\begin{subfigure}{0.2\textwidth}
\centering
\includegraphics[height=2.1cm]{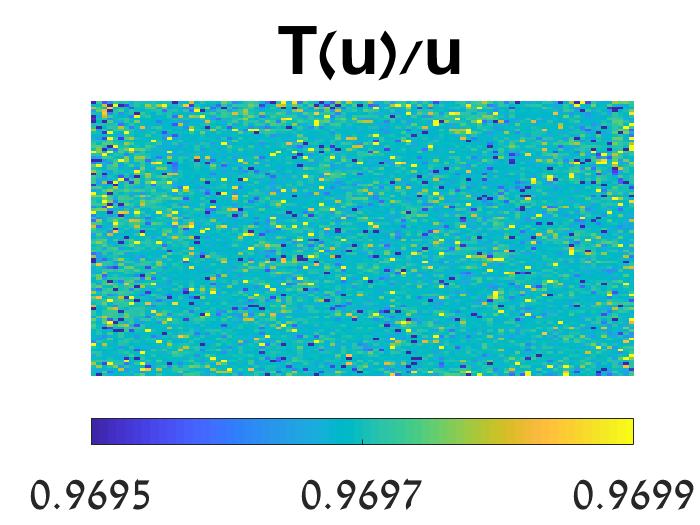}
\caption{Constant Values of $\frac{T(u)}{u}$}
\end{subfigure}
\caption{EPLL texture generator eigenfunction $u$, with $\lambda=0.9697$. We validate this is an eigenfunction following Propositions \ref{prop3}-\ref{prop4}. Also, $\frac{T(u)}{u}=C$ ($96\%$ of values are in the displayed range).}
\label{Fig::EPLL_noise_gen}
\end{figure*}

Fig. \ref{Fig::degradation} 
demonstrates eigenfunction robustness to various small degradations. The Rayleigh quotient of a degraded eigenfunction is similar to that of a "clean" one, and power iterations make it converge to the critical point in the vicinity - the original eigenvalue. Especially note the "small message" robustness property, that will be used for the decryption-encryption application (Sec. \ref{sec::app}). Note that noise robustness holds in a very wide sense. For example, the denoiser considers textures and fine details as noise, which can be removed. The texture generator, on the other hand, prefers noise and textures, and thus considers coarse structures as noise. The decay profiles of a degraded eigenfunction also exhibit a distinct pattern (Fig \ref{Fig::decay_noise}), similar to those of the "clean" eigenfunction, but sometimes distorted in the beginning. 

\begin{figure*}
\begin{tabular}{c}
\rotatebox{90}{\hspace{1cm}{degraded}}\\
\rotatebox{90}{\hspace{1cm}{corrected}}\\
\rotatebox{90}{\hspace{0.3cm}{difference}}
\end{tabular}
\captionsetup[subfigure]{justification=centering}
\begin{subfigure}{0.15\textwidth}
\centering
\begin{tabular}{c}
\includegraphics[width=0.9\textwidth]{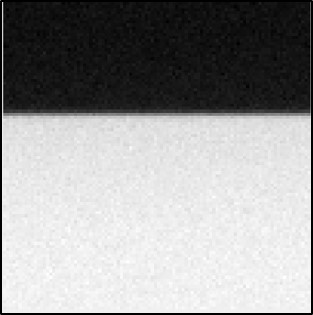} \\
\includegraphics[width=0.9\textwidth]{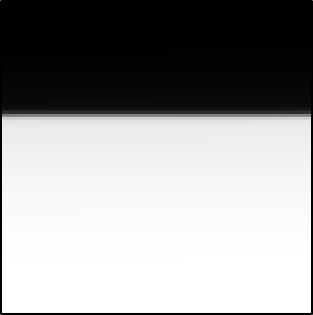} \\
\includegraphics[width=0.9\textwidth]{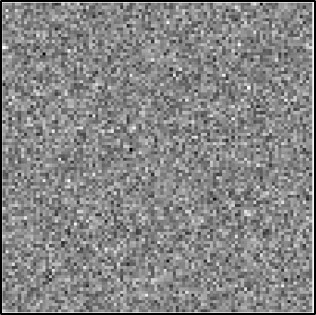}
\end{tabular}
\caption{\\Gaussian noise}
\end{subfigure}
\begin{subfigure}{0.15\textwidth}
\centering
\begin{tabular}{c}
\includegraphics[width=0.9\textwidth]{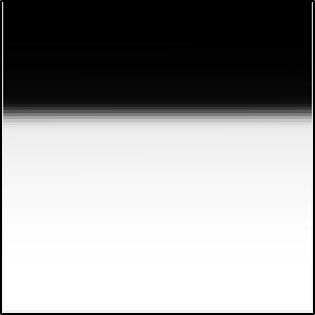} \\
\includegraphics[width=0.9\textwidth]{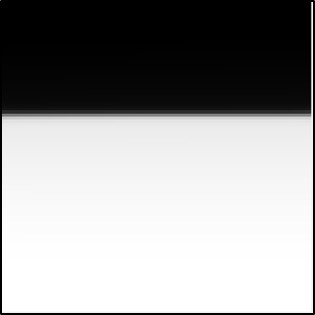} \\
\includegraphics[width=0.9\textwidth]{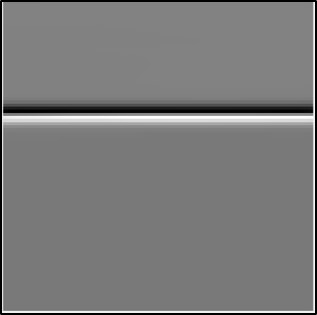}
\end{tabular}
\caption{\\Gaussian blur}
\end{subfigure}
\begin{subfigure}{0.15\textwidth}
\centering
\begin{tabular}{c}
\includegraphics[width=0.9\textwidth]{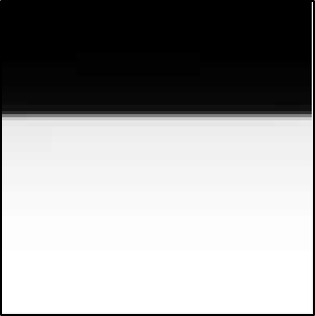} \\
\includegraphics[width=0.9\textwidth]{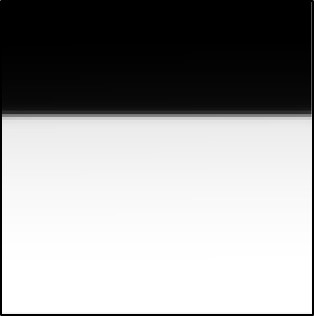} \\
\includegraphics[width=0.9\textwidth]{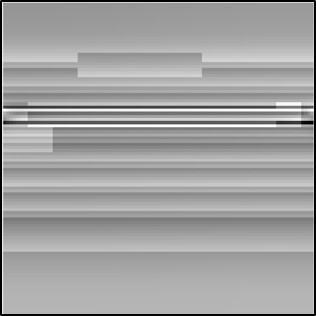}
\end{tabular}
\caption{JPEG compression}
\end{subfigure}
\begin{subfigure}{0.15\textwidth}
\centering
\begin{tabular}{c}
\includegraphics[width=0.9\textwidth]{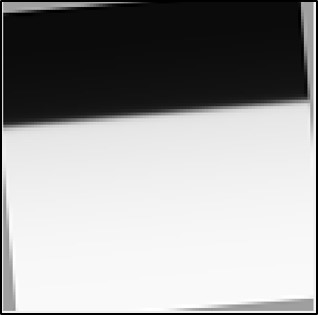} \\
\includegraphics[width=0.9\textwidth]{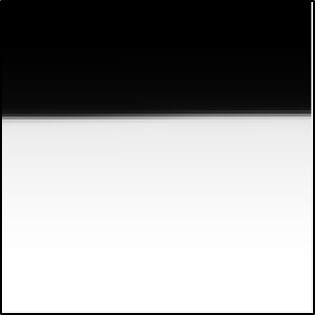} \\
\includegraphics[width=0.9\textwidth]{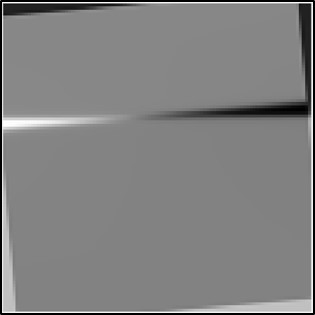}
\end{tabular}
\caption{\\Rotation, $5^{\circ}$}
\end{subfigure}
\begin{subfigure}{0.15\textwidth}
\centering
\begin{tabular}{c}
\includegraphics[width=0.9\textwidth]{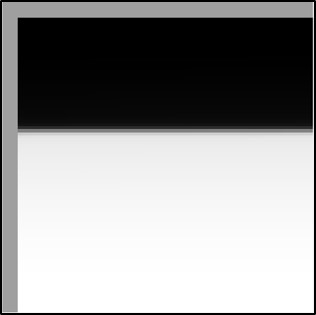} \\
\includegraphics[width=0.9\textwidth]{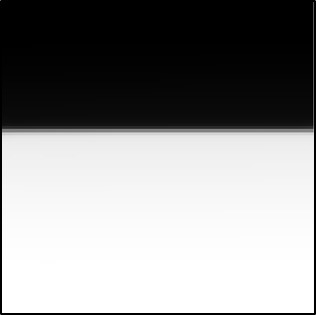} \\
\includegraphics[width=0.9\textwidth]{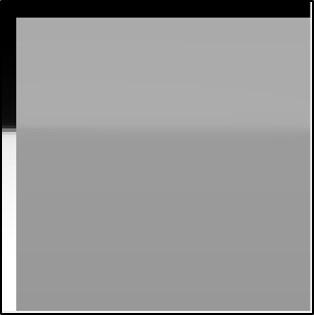}
\end{tabular}
\caption{\\Shift, $5X5$ pixels}
\end{subfigure}
\begin{subfigure}{0.15\textwidth}
\centering
\begin{tabular}{c}
\includegraphics[width=0.9\textwidth]{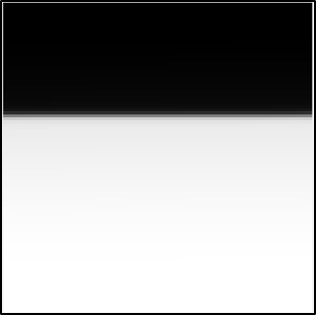} \\
\includegraphics[width=0.9\textwidth]{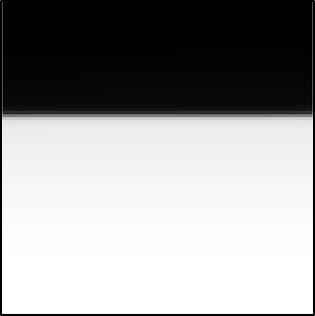} \\
\includegraphics[width=0.9\textwidth]{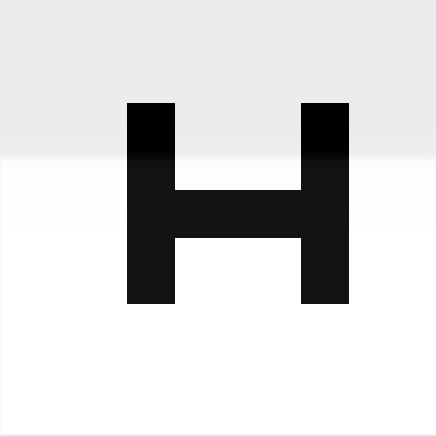}
\end{tabular}
\caption{Small message added}
\label{Fig::added_message}
\end{subfigure}
\begin{subfigure}{0.24\textwidth}
\centering
\includegraphics[width=0.6\textwidth]{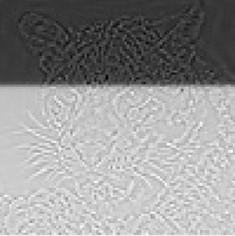}
\caption{Large + small \\eigenfunctions}
\end{subfigure}
\begin{subfigure}{0.24\textwidth}
\centering
\includegraphics[width=0.6\textwidth]{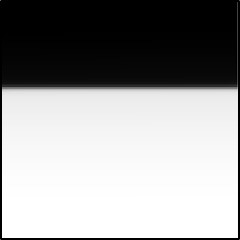}
\caption{After small eigenfunction is removed from (g)}
\end{subfigure}
\begin{subfigure}{0.24\textwidth}
\centering
\includegraphics[width=0.6\textwidth]{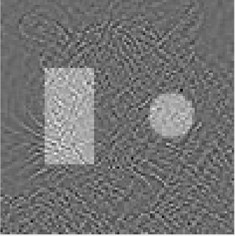}
\caption{Small eigenfunction + structure}
\end{subfigure}
\begin{subfigure}{0.24\textwidth}
\centering
\includegraphics[width=0.6\textwidth]{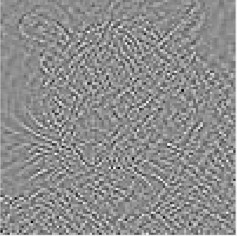}
\caption{After structure is removed\\from (i)}
\end{subfigure}
\caption{(a)-(f): Degradation robustness of EPLL large eigenfunction for various small degradations. Top: degraded eigenfunctions (e.g. noise of $\sigma=0.01$, blur of $\sigma=0.01$, see different columns). Middle: corrected eigenfunctions after applying 500 power iterations. Bottom: Difference images. (g)-(h): A small eigenfunction, added to the large, is considered by the denoiser as "noise", and thus removed. (i)-(j): A structure, added to the small eigenfunction, is considered by the texture generator as "noise", and thus removed.}
\label{Fig::degradation}
\end{figure*}

\subsection{Application: Encryption-Decryption Scheme}\label{sec::app}
We suggest an encryption-decryption scheme and demonstrate it for the EPLL denoiser. We exploit eigenfunction robustness to adding a "small message" (Fig. \ref{Fig::added_message}), which can be removed by applying power iterations. Then, the difference image will reveal the message, which may be too small to be detected in the original image. We also exploit the strong impact of parameters chosen to generate the eigenfunction (Sec. \ref{Sec::method}), which can be decided by the sender and secretly shared with the receiver. However, as they are unknown to enemies, decryption is practically impossible for them.

Fig. \ref{Fig::app_flow} 
shows the three components of the scheme. First, the sender decides on different impacting factors, generates the eigenfunction, and adds the secret hidden message. Second, the receiver uses the secret impacting factors to apply power iterations to remove the message. Some simple post-processing may also be needed. Third, enemies try to decrypt the hidden message using the power iteration with unknown factors. More specifically, the prior model for EPLL, GMM, offers numerous options of covariance matrices, mixing weights and component numbers, which are practically impossible to guess. Thus the enemies fail, and the message remains hidden. Fig. \ref{Fig::app_examples} 
shows several examples of the application.

\begin{figure*}
\captionsetup[subfigure]{justification=centering}
\centering
\includegraphics[width=0.9\textwidth]{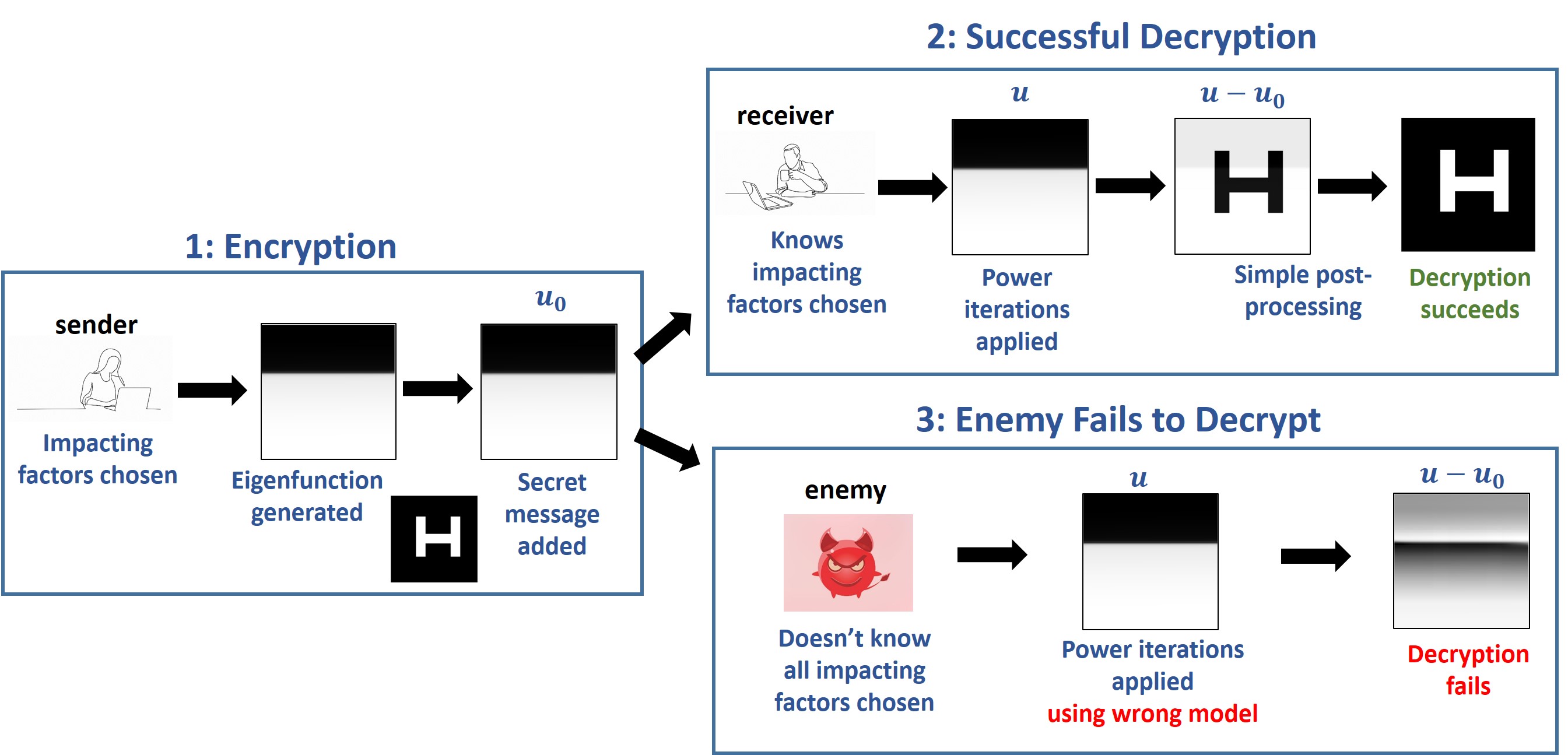}
\caption{Flow chart: encryption-decryption application. The receiver applies power iterations with secret impacting factors to decrypt the message. However, enemies apply power iterations with unknown factors and thus fail to decrypt.}
\label{Fig::app_flow}
\end{figure*}

\begin{figure*}
\captionsetup[subfigure]{justification=centering}
\centering
\begin{subfigure}{0.18\textwidth}
\centering
\includegraphics[width=0.8\textwidth]{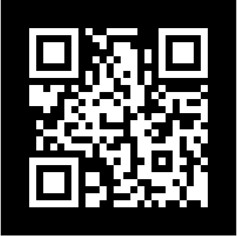}
\caption{Secret message to be added}
\end{subfigure}
\begin{subfigure}{0.18\textwidth}
\centering
\includegraphics[width=0.8\textwidth]{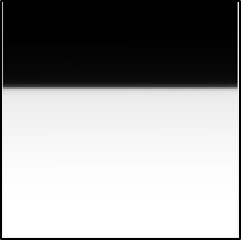}
\caption{Large eigenfunction + message, $u_0$}
\end{subfigure}
\begin{subfigure}{0.18\textwidth}
\centering
\includegraphics[width=0.8\textwidth]{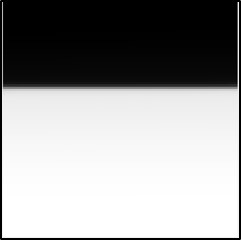}
\caption{Power iterations applied, $u$}
\end{subfigure}
\begin{subfigure}{0.18\textwidth}
\centering
\includegraphics[width=0.8\textwidth]{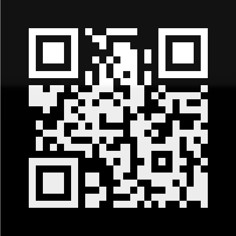}
\caption{\\Difference image, $u-u_0$}
\end{subfigure}
\\
\begin{subfigure}{0.18\textwidth}
\centering
\includegraphics[width=0.8\textwidth]{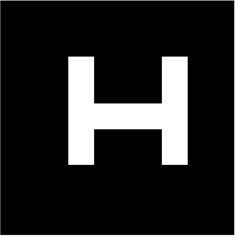}
\caption{Secret message to be added}
\end{subfigure}
\begin{subfigure}{0.18\textwidth}
\centering
\includegraphics[width=0.8\textwidth]{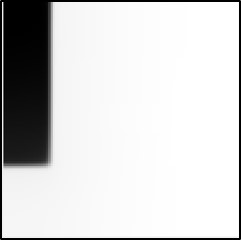}
\caption{$2^{nd}$ eigenfunction + message, $u_0$}
\end{subfigure}
\begin{subfigure}{0.18\textwidth}
\centering
\includegraphics[width=0.8\textwidth]{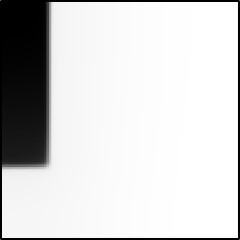}
\caption{Power iterations applied, $u$}
\end{subfigure}
\begin{subfigure}{0.18\textwidth}
\centering
\includegraphics[width=0.8\textwidth]{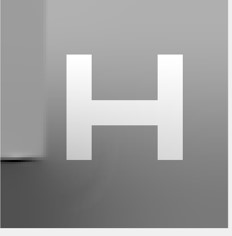}
\caption{\\Difference image, $u-u_0$}
\end{subfigure}
\begin{subfigure}{0.18\textwidth}
\centering
\includegraphics[width=0.8\textwidth]{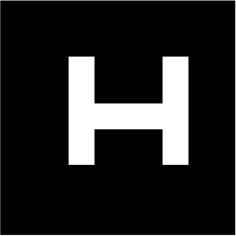}
\caption{Post-processing of $u-u_0$}
\end{subfigure}
\begin{subfigure}{0.18\textwidth}
\centering
\includegraphics[width=0.8\textwidth]{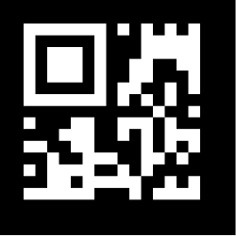}
\caption{Secret message to be added}
\end{subfigure}
\begin{subfigure}{0.18\textwidth}
\centering
\includegraphics[width=0.8\textwidth]{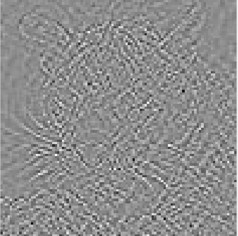}
\caption{Small eigenfunction + message, $u_0$}
\end{subfigure}
\begin{subfigure}{0.18\textwidth}
\centering
\includegraphics[width=0.8\textwidth]{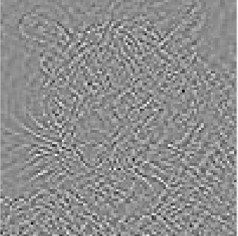}
\caption{Power iterations applied, $u$}
\end{subfigure}
\begin{subfigure}{0.18\textwidth}
\centering
\includegraphics[width=0.8\textwidth]{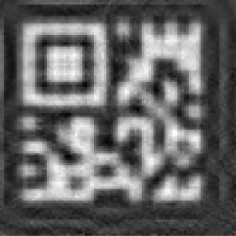}
\caption{\\Difference image, $u-u_0$}
\end{subfigure}
\begin{subfigure}{0.18\textwidth}
\centering
\includegraphics[width=0.8\textwidth]{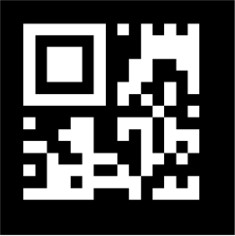}
\caption{Post-processing of $u-u_0$}
\end{subfigure}
\caption{Examples of the encryption-decryption scheme for different messages and EPLL eigenfunctions. For each row, from left to right: a secret message; an eigenfunction with message added ($u_0$); after (500) power iterations applied ($u$); and their difference ($u-u_0$). Row 2: we use a "processed" eigenfunction (see Sec. \ref{sec::more_ef}) and perform simple segmentation as post-processing. Row 3: we perform the following post-processing: locating the large square template of the message in the result, using a correlation matrix, then thresholding the central pixel of each image cell to determine its value.}
\label{Fig::app_examples}
\end{figure*}

\subsection{Deblurring Operators}
To the best of our knowledge, this is the first attempt to find eigenpairs for non-convex deblurring operators (but see ~\cite{schmidt2018inverse},~\cite{bungert2019solution}). Lacking previous theory, we use knowledge of sharpening operators and specific priors and our analysis in Sec. \ref{Sec::theory} for analysis. The possible meaning in this context of a complementary operator, small eigenfunction, or equivalence to decay profiles, remains for now an open question. 

Both deblurring operators examined, TV and EPLL, can be formulated as optimization problems similar to Eq. \ref{Eq::TV}, \ref{Eq::EPLL}, respectively, with an adapted fidelity term: $\|f-Au\|^2$. We also note that we slightly update the power iteration scheme: mean is kept as $\overline{f^0}$ for $f^0$ and throughout evolution. 

Fig. \ref{Fig::deblur} shows results for both operators. We show different large eigenfunctions of the EPLL deblurring operator, generated using different motion blur kernels. Note that the straight but slightly rounded shapes of eigenfunctions correspond to the different kernel shapes, and are in accordance with the observations made in~\cite{shaham2016visualizing}, regarding the behavior of the EPLL denoiser. We also show a large eigenfunction of the TV deblurring operator, generated using a Gaussian blur kernel. Note that the eigenfunction shape is in accordance with the convex nature of eigenfunctions of the TV denoiser~\cite{gilboa2018beyond}. For both operators, the eigenfunctions are textural objects - in accordance with the detail-enhancing nature of deblurring operators. Also, the resulting eigenvalues are larger than 1, as expected from detail-enhancing operators, and in accordance with Corollary \ref{cor_lam}.  

\begin{figure}
\captionsetup[subfigure]{justification=centering}
\centering
\begin{subfigure}{0.19\textwidth}
\centering
    \includegraphics[width=0.6\textwidth]{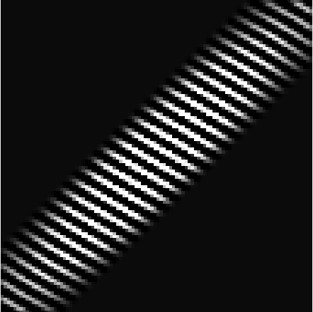}
    \caption{EPLL eigenfunction}
\end{subfigure}
\begin{subfigure}{0.19\textwidth}
\centering
    \includegraphics[width=0.6\textwidth]{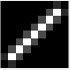}
    \caption{Blur kernel for (a)}
\end{subfigure}
\\
\begin{subfigure}{0.15\textwidth}
\centering
    \includegraphics[height=2cm]{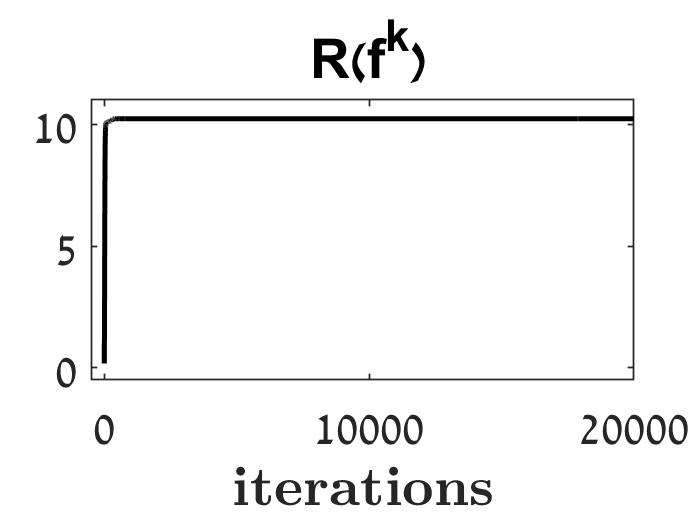}
    \caption{Rayleigh quotient $\rightarrow \lambda$}
\end{subfigure}
\begin{subfigure}{0.15\textwidth}
\centering
    \includegraphics[height=2cm]{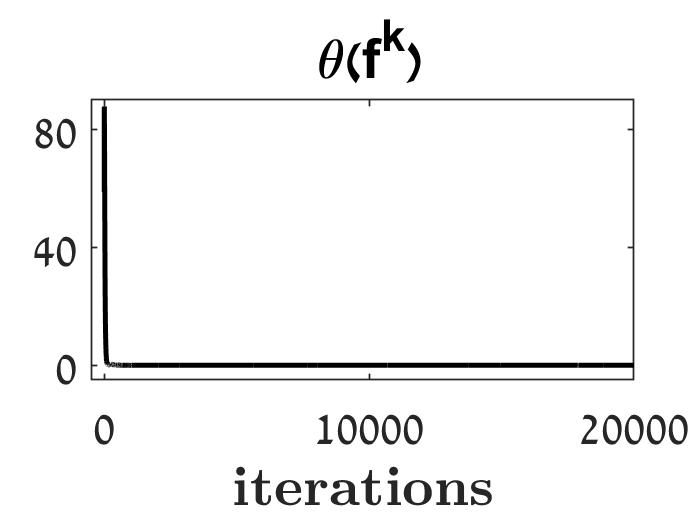}
    \caption{$\theta$ between $u^k, T(u^k) \rightarrow 0^{\circ}$}
\end{subfigure}
\begin{subfigure}{0.15\textwidth}
\centering
    \includegraphics[height=2cm]{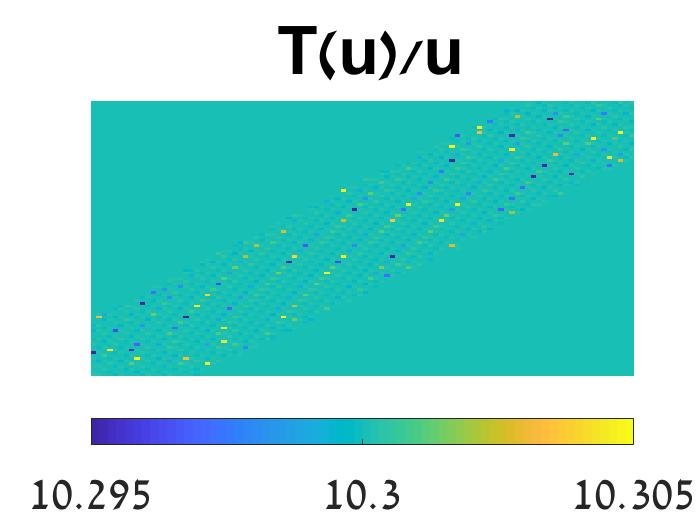}
    \caption{Constant Values of $\frac{T(u)}{u}$}
\end{subfigure}
\\
\begin{subfigure}{0.15\textwidth}
\centering
    \includegraphics[width=0.8\textwidth]{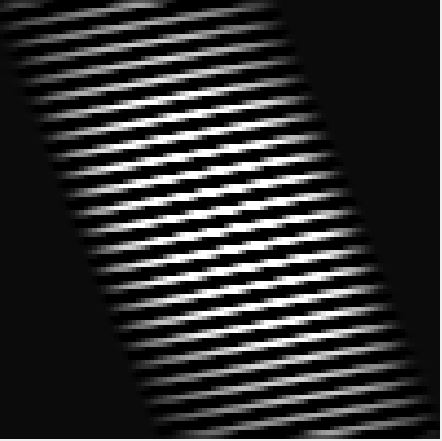}
    \caption{EPLL eigenfunction}
\end{subfigure}
\begin{subfigure}{0.15\textwidth}
\centering
    \includegraphics[width=0.37\textwidth]{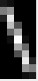}
    \caption{Blur kernel for (f)}
\end{subfigure}
\begin{subfigure}{0.15\textwidth}
\centering
    \includegraphics[width=0.8\textwidth]{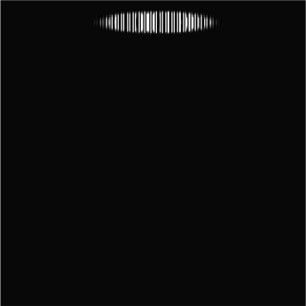}
    \caption{TV eigenfunction}
\end{subfigure}
\caption{Eigenfunctions of deblurring operators. For the EPLL deblurring operator: eigenfunction (a) with eigenvalue $\lambda=10.3001$ for motion kernel (b) (we validate this is an eigenfunction following Propositions \ref{prop3}-\ref{prop4}), eigenfunction (f) with eigenvalue $\lambda=10.3845$ for kernel (g). For the TV deblurring operator: eigenfunction (h) with $\lambda=16.0913$.}
\label{Fig::deblur}
\end{figure}

\section{Conclusion}\label{Sec::conc}
We suggest a generalized method for solving and analyzing generic eigenproblems by adapting the well established power iteration. We handle non-linear, black-box operators, induced by image processing algorithms, where no analytic operator is at hand. Such eigenproblems are very interesting, as they reveal the stable and unstable modes of the operator (its most- and least-suitable inputs). We show steady state properties of the process, as well as convergence for contraction operators (Lipschitz continuous). We also present analysis tools for validation. We demonstrate our method for two image denoisers: the well-known functional-induced total-variation, and the black-box EPLL (based on natural image statistics). We find eigenfunctions with large and small eigenvalues, and also several of them in an iterative process, and demonstrate their robustness to various degradations. Based on this insight we suggest an encryption-decryption application. Finally, we analyze eigenproblems for generic deblurring operators, apparently for the first time. 

\section*{Acknowledgment}
We acknowledge support by the Israel Science Foundation (grant No. 534/19) and the Technion Ollendorff Minerva Center.

\ifCLASSOPTIONcaptionsoff
  \newpage
\fi

\bibliographystyle{IEEEtran}
\bibliography{main}

\begin{IEEEbiographynophoto}{Ester Hait-Fraenkel} 
is pursuing her Ph.D. in Electrical Engineering in the Technion - Israel Institute of Technology, where she received her B.Sc. (Cum Laude) and M.Sc. in Electrical Engineering in 2014 and 2016, respectively. Her research interests include image processing and computer vision.
\end{IEEEbiographynophoto}


\begin{IEEEbiographynophoto}{Guy Gilboa}
received his Ph.D. from the Electrical Engineering Department, Technion - Israel Institute of Technology in 2004. He was a postdoctoral fellow with UCLA and had various development and research roles with Microsoft and Philips Healthcare. Since 2013 he is a faculty member with the Electrical Engineering Department, Technion - Israel Institute of Technology. He has authored some highly cited papers on topics such as image sharpening and denoising, nonlocal operators theory, and texture analysis. He received several prizes, including the Eshkol Prize by the Israeli Ministry of Science, the Vatat Scholarship, and the Gutwirth Prize. He serves at the editorial boards of the journals IEEE SPL, JMIV and CVIU.
\end{IEEEbiographynophoto}

\end{document}